\documentclass[11pt,english]{smfart}

\usepackage{amsmath,leftidx,amsthm,sfheaders}
\usepackage[all]{xy}
\usepackage{xspace,smfthm}
\usepackage[psamsfonts]{amssymb}
\usepackage[latin1]{inputenc}
\usepackage{graphicx,color}
\usepackage{hyperref,fancyhdr}

\newtheorem{theorem}{\bf Theorem}[section]
\newtheorem{proposition}[theorem]{\bf Proposition}
\newtheorem{definition}[theorem]{\bf Definition}
\newtheorem{Theorem}{\bf Theorem}
\newtheorem*{claim}{\bf Claim}
\newtheorem{lemma}[theorem]{\bf Lemma}
\newtheorem{corollary}[theorem]{\bf Corollary}

\def\C{{\mathbb C}}

\def\R{{\mathbb R}}

\def\B{\mathbb{B}}
\def\D{\mathbb{D}}

\def\J{\mathcal{J}}
\def\K{\mathcal{K}}
\def\p{\mathbb{P}}

\def\supp{\textup{supp}}

\def\pe{\textup{ := }}

\def\Per{\textup{Per}}
\def\bif{\textup{bif}}

\def\Mand{\mathbf{M}}
\def\J{\mathcal{J}}

\def\ud{\underline{d}}

\def\and{{\quad\text{and}\quad}}

\title{Equidistribution towards the bifurcation current I: Multipliers and degree $d$ polynomials}
\author{Thomas Gauthier}
\address{LAMFA, Universit\'e de Picardie Jules Verne, 33 rue Saint-Leu, 80039 AMIENS Cedex 1, FRANCE}
\email{thomas.gauthier@u-picardie.fr}
\begin{document}

\begin{abstract}
In the moduli space $\mathcal{P}_d$ of degree $d$ polynomials, the set $\Per_n(w)$ of classes $[f]$ for which $f$ admits a cycle of exact period $n$ and multiplier multiplier $w$ is known to be an algebraic hypersurface. We prove that, given $w\in\C$, these hypersurfaces equidistribute towards the bifurcation current as $n$ tends to infinity.
\end{abstract}

\maketitle

\tableofcontents

%\fancyhead[CO]{Equidistribution towards the bifurcation current}

\section*{Introduction}

In a holomorphic family $(f_\lambda)_{\lambda\in\Lambda}$ of degree $d\geq2$ rational maps, the \emph{bifurcation locus} is the closure in the parameter space $\Lambda$ of the set of discontinuity of the map $\lambda\mapsto \J_\lambda$, where $\J_{\lambda}$ is the Julia set of $f_{\lambda}$. The study of the global geography of the parameter space $\Lambda$ is related to the study of the hypersurfaces
\begin{center}
$\Per_n(w)\pe\{\lambda\in\Lambda$ s.t. $f_\lambda$ has a $n$-cycle of multiplier $w\}~.$
\end{center}
In their seminal work \cite{MSS}, Ma\~{n}\'e, Sad and Sullivan prove that the bifurcation locus is nowhere dense in $\Lambda$ and coincides with the closure of the set of parameters for which $f_{\lambda}$ admits a non-persistent neutral cycle (see also \cite{lyubich}). In particular, by Montel's Theorem, this implies that any bifurcation parameter can be approximated by parameters with a super-attracting periodic point, i.e. the bifurcation locus is contained in the closure of the set $\bigcup_{n\geq1}\Per_n(0)$.

~

DeMarco proved that, in any holomorphic family, the bifurcation locus can be naturally endowed with a closed positive $(1,1)-$current $T_\bif$, called the bifurcation current (see e.g.  \cite{DeMarco1}). This current may be defined as $dd^cL$ where $L$ is the continuous plurisubharmonic function which sends a parameter $\lambda$ to the Lyapunov exponent $L(\lambda)=\int_{\p^1}\log|f_\lambda'|\,\mu_\lambda$ of $f_\lambda$ with respect to its maximal entropy measure $\mu_\lambda$. The current $T_\bif$ provides an appropriate tool for studying bifurcations from a measure-theoretic viewpoint. When $\dim_\C\Lambda=\kappa\geq2$, it gives rise to a positive measure $\mu_\bif:=T_\bif^\kappa=T_\bif\wedge\cdots\wedge T_\bif$ called the \emph{bifurcation measure} which, in a certain way, detects maximal bifurcations that arise in the family $(f_\lambda)_{\lambda\in\Lambda}$.

It appears that, when we fix $w\in\C$, the current $T_\bif$ is very related to the asymptotic distribution of the hypersurfaces $\Per_n(w)$, as $n\to\infty$. Indeed, Bassanelli and Berteloot proved that
\begin{center}
$d^{-n}[\Per_n(w)]\longrightarrow_{n\to\infty}T_\bif$
\end{center}
 for a given $|w|<1$ in the weak sense of currents, using the fact that the function $L$ is a global potential of $T_\bif$ in any holomorphic family (see \cite{BB3}). We refer the reader to the survey \cite{dsurvey} or the lecture notes \cite{bsurvey} for a report on recent results involving bifurcation currents and further references.

~

Let us now focus on the case of the moduli space $\mathcal{P}_d$ of degree $d$ polynomials with $d-1$ marked critical points, i.e. the set of affine conjugacy classes of degree $d$ polynomials with $d-1$ marked critical points. Notice that it, in that family, the bifurcation measure has finite mass and is supported by the Shilov boundary of the \emph{connectedness locus}:
\[\mathcal{C}_d:=\{[P]\in\mathcal{P}_d\, ; \ \J_P \ \text{is connected}\}~,\]
which is a compact subset of $\mathcal{P}_d$. In the present case, Bassanelli and Berteloot \cite{BB2} prove that this convergence also holds when $|w|=1$. In the present paper, we prove that this actually holds for any $w\in\C$. In future works, we shall investigate equidistribution properties of higher codimension algebraic varieties in $\mathcal{P}_d$ defined by intersections of hypersurfaces $\Per_n(w)$, or defined by the persistence of critical orbit relations.

Our main result can be stated as follows.

\begin{Theorem}
Let $d\geq2$ and $w\in\C$ be any complex number. Then the sequence $d^{-n}[\Per_n(w)]$ converges in the weak sense of currents to the bifurcation current $T_\bif$ in the moduli space $\mathcal{P}_d$ of degree $d$ polynomials with $d-1$ marked critical points.\label{tmconv}
\end{Theorem}

Notice that, when $d=2$, the moduli space of quadratic polynomials with one marked critical point is isomorphic to the quadratic family $(z^2+c)_{c\in\C}$ and that, in the quadratic family, this result is a particular case of the main Theorem of \cite{multipliers}. Notice also that for $d\geq3$, up to a finite branched covering, $\mathcal{P}_d$ is isomorphic to $\C^{d-1}$.

~

Let us now sketch the strategy of the proof of Theorem \ref{tmconv} developed in \cite{multipliers} in the quadratic case and then explain how to adapt it to our situation. It is known that there exists a global potential $\varphi_n$ of the current $d^{-n}[\Per_n(w)]$ that converges, up to taking a subsequence, in $L^1_\textup{loc}$ to a psh function $\varphi\leq L$ which satisfies $\varphi=L$ on hyperbolic components (see \cite{BB3}).

In the quadratic case, the bifurcation locus is the boundary of the Mandelbrot set $\Mand\Subset\C$ and $\C\setminus \Mand$ is a hyperbolic component, hence $\varphi=L$ outside $\Mand$. First, we explain why the positive measure $\Delta L$ of the Mandelbrot set doesn't give mass to the boundary of connected components of the interior of $\Mand$. Secondly, we establish a comparison lemma for subharmonic function which, in that case, gives $\varphi=L$ and the proof is complete.

To adapt the proof to the situation $d\geq3$, we first establish a generalization of the comparison principle for plurisubharmonic functions. Again, it is known that $\varphi=L$ on the escape locus and we shall use the comparison principle recursively on the number of critical points of bounded orbits in suitable local subvarieties of $\mathcal{P}_d$.

~

Let us mention that the comparison principle we prove may be of independant interest. In contrast to the classical domination Theorem of Bedford and Taylor (see e.g. \cite{BedfordTaylor}), we don't need to be able to compare the Monge-Amp\`ere masses of two psh functions to compare the functions themselves. Precisely, we prove the following which is a generalization in higher dimension of \cite[Lemma 3]{multipliers}.

\begin{Theorem}[Comparison principle]\label{tmcompXav}
Let $\mathcal{X}$ be a complex manifold of dimension $k\geq1$. Assume that there exists a smooth psh function $w$ on $\mathcal{X}$ and a strict analytic subset $\mathcal{Z}$ of $\mathcal{X}$ such that $(dd^cw)^k$ is a non-degenerate volume form on $\mathcal{X}\setminus\mathcal{Z}$. Let $\Omega\subset\mathcal{X}$ be a domain of $\mathcal{X}$ with $\mathcal{C}^1$ boundary and let $u,v\in\mathcal{PSH}(\Omega)$ and $K\Subset\Omega$ be a compact set. Assume that the following assumptions are satisfied:
\begin{itemize}
\item $v$ is continuous, $\supp((dd^cv)^k)\subset \partial K$ and $(dd^cv)^k$ has finite mass,
\item for any connected component $U$ of $\mathring{K}$, $(dd^cv)^k(\partial U)=0$,
\item $u\leq v$ on $\Omega$ and $u=v$ on $\Omega\setminus K$.
\end{itemize}
Then $u=v$ on $\Omega$.
\end{Theorem}

Our strategy to apply Theorem \ref{tmcompXav} relies on describing (partially) the currents $T_\bif^k$ in restriction to suitable local analytic subvarieties of the moduli space $\mathcal{P}_d$. When $1\leq k\leq d-2$, for any parameter $\lambda_0$ lying in an open dense subset of $\mathcal{P}_d\setminus\mathcal{C}_d$, we build a local analytic subvariety passing through $\lambda_0$ and in restriction to which the bifurcation measure enjoys good properties. The proof relies on techniques developped in the context of horizontal-like maps (see \cite{DDS,dujardin2}). This is the subject of the following result.

\begin{Theorem}
Pick $d\geq3$. There exists an open dense subset $\Omega\subset\mathcal{P}_d\setminus\mathcal{C}_d$ such that for any $\{P\}\in\Omega$, if $0\leq k\leq d-2$ is the number of critical points of $P$ with bounded orbit, then there exists an analytic set $\mathcal{X}_0\subset \Omega$, a complex manifold $\mathcal{X}$ of dimension $k$ and a finite holomorphic map $\pi:\mathcal{X}\to\mathcal{X}_0$ such that $\{P\}\in \mathcal{X}_0$ and
\begin{enumerate}
\item the measure $\mu_\mathcal{X}:=\pi^*(T_\bif^k|_{\mathcal{X}_0})$ is a compactly supported finite measure on $\mathcal{X}$,
\item for any relatively compact connected component $\mathcal{U}$ of the open set $\mathcal{X}\setminus\supp(\mu_\mathcal{X})$,
\begin{center}
$\mu_\mathcal{X}(\partial\mathcal{U})=0$,
\end{center}
\item if $\{Q\}$ lies in the non-relatively compact connected component of $\mathcal{X}\setminus\supp(\mu_\mathcal{X})$, then the degree $d$ polynomial $Q$ has at most $k-1$ critical points with bounded orbit.
\end{enumerate}
\label{tm:locprop}
\end{Theorem}
The proof of Theorem \ref{tm:locprop} is the combination of Theorem \ref{cor:TI} and Claim of Section \ref{sec:distrib}.

~

The last step of the proof of Theorem \ref{tmconv} consists in applying the comparison Theorem in $\mathcal{P}_d$ for $K=\mathcal{C}_d$. To this aim, we need to prove that the bifurcation measure $\mu_\bif$ does not give mass to the boundary of components of the interior of $\mathcal{C}_d$. Building on the description of the bifurcation measure given by Dujardin and Favre \cite{favredujardin} and properties of invariant line fields established by McMullen \cite{McMullen}, we prove the following.

\begin{Theorem}\label{lmboundary}
Let $\mathcal{U}\subset\mathcal{P}_d$ be any connected component the interior of $\mathcal{C}_d$. Then
$$\mu_\bif(\partial \mathcal{U})=0.$$
\end{Theorem}

Let us finally explain the organization of the paper. Section 1 is devoted to required preliminaries. In Section 2, we establish our comparison principle for psh functions. In Section 3, we prove a slightly more precise version of Theorem~\ref{tm:locprop}. Section 4 is concerned with the proof of Theorem~\ref{lmboundary}. Finally, we give the proof of Theorem~\ref{tmconv} in Section 5.

\section{Preliminaries}
\subsection{A good parametrization of $\mathcal{P}_d$, $d\geq3$}

Recall that if $(f_\lambda)_{\lambda\in\Lambda}$ is a holomorphic family of polynomials, we say that  $\Lambda$ is with $d-1$ marked critical points if there exists holomorphic maps $c_1,\ldots,c_{d-1}:\Lambda\to\C$ such that $C(f_\lambda)=\{c_1(\lambda),\ldots,c_{d-1}(\lambda)\}$ counted with multiplicity.

It is now classical that the moduli space of degree $d$ polynomials with $d-1$ marked critical points, i.e. the space of degree $d$ polynomials with $d-1$ marked critical points modulo affine conjugacy, is a complex orbifold of dimension $d-1$ which is not smooth when $d\geq3$. Here, we shall use the following parametrization
$$P_{c,a}(z):= \frac{1}{d}z^d+\sum_{j=2}^{d-1}(-1)^{d-j}\sigma_{d-j}(c)\frac{z^j}{j}+a^d,$$
where $\sigma_j(c)$ is the monic symmetric polynomial in $(c_1,\ldots,c_{d-2})$ of degree $j$. Observe that the critical points of $P_{c,a}$ are exactly $c_0$, $c_1, \ldots , c_{d-2}$ with the convention that $c_0:=0$, and that the canonical projection $\pi:\C^{d-1}\longrightarrow\mathcal{P}_d$  which maps $(c_1,\ldots,c_{d-2},a) \in\C^{d-1}$ to the class of $P_{c,a}$ in $\mathcal{P}_d$ is  $d(d-1)$-to-one (see \cite[\S 5]{favredujardin}).

Recall that the \emph{Green function} of $P_{c,a}$ is the subharmonic function defined for $z\in\C$ by
$$g_{c,a}(z):=\lim_{n\to\infty}d^{-n}\log\max\left(1,|P_{c,a}^n(z)|\right)~,$$
and that the filled-in Julia set of $P_{c,a}$ is the compact subset of $\C$
$$\K_{c,a}:=\{z\in\C \ | \ (P_{c,a}^n(z))_{n\geq1}\text{ is bounded in }\C\}~.$$
Remark that $\K_{c,a}=\{z\in\C \ | \ g_{c,a}(z)=0\}$. Recall also that the chaotic part of the dynamics is supported by the \emph{Julia set} $\J_{c,a}=\partial\K_{c,a}$ of $P_{c,a}$. The function $(c,a,z)\in\C^d\mapsto g_{c,a}(z)$ is actually a non-negative plurisubharmonic continuous function on $\C^d$. We set
$$\mathcal{B}_i:=\{(c,a)\in\C^{d-1} \ | \ c_i\in \K_{c,a}\}=\{(c,a)\in\C^{d-1} \ | \ g_{c,a}(c_i)=0\}~$$
and $\mathcal{C}_d:=\{(c,a)\in\C^{d-1}\ | \ \max_{0\leq i\leq d-2}\left(g_{c,a}(c_i)\right)=0\}=\bigcap_i\mathcal{B}_i$. It is known that $\K_{c,a}$ is connected if and only if $(c,a)\in\mathcal{C}_d$. Let us finally set
$$H_\infty:=\p^{d-1}(\C)\setminus\C^{d-1}=\{[c:a:0]\in\p^{d-1}(\C)\} \ \text{and} \ H_i:=\{[c:a:0]\in H_\infty \ : \ P_{c,a}(c_i)=0\}~.$$
We shall use the following which has been established by Basanelli and Berteloot, relying on previous works by Branner and Hubbard \cite{BH} and by Dujardin and Favre \cite{favredujardin} (see \cite[Lemma 4.1 $\&$ Theorem 4.2]{BB2}):

\begin{theorem}[Bassanelli-Berteloot, Branner-Hubbard, Dujardin-Favre]\label{tmBH}
\begin{enumerate}
\item For any $0\leq i\leq d-2$, the cluster set of $\mathcal{B}_i$ in $\p^{d-1}(\C)$ coincides with $H_i$,
\item For any $1\leq k\leq d-2$ and for any $k$-tuple $0\leq i_1<\cdots<i_k\leq d-2$, the cluster set of $\bigcap_{j=1}^k\mathcal{B}_{i_j}$ in $\p^{d-1}(\C)$, which is exactly $\bigcap_{j=1}^kH_{i_j}$, is a pure $(d-2-k)$-dimensional algebraic variety of $H_\infty$,
\item The set $\mathcal{C}_d$ is a compact connected subset of $\C^{d-1}$.
\end{enumerate}
\end{theorem}

\subsection{The bifurcation current}
Classically, a parameter $(c_0,a_0)\in\C^{d-1}$ is said $\J$-\emph{stable} if there exists an open neighborhood $U\subset\C^{d-1}$ of $(c_0,a_0)$ such that for any $(c,a)\in U$, there exists a homeomorphism $\psi_{c,a}:\J_{c_0,a_0}\to\J_{c,a}$ which conjugates $P_{c_0,a_0}$ to $P_{c,a}$, i.e. such that
$$\psi_{c,a}\circ P_{c_0,a_0}(z)=P_{c,a}\circ\psi_{c,a}(z), \ z\in\J_{c_0,a_0}~.$$
The \emph{stability locus} $\mathcal{S}$ of the family $(P_{c,a})_{(c,a)\in\C^{d-1}}$ is the set of $\J$-stable parameters and the \emph{bifurcation locus} is its complement $\C^{d-1}\setminus\mathcal{S}$.
\begin{definition}
We say that the critical point $c_i$ is \emph{passive} at $(c_0,a_0)\in\C^{d-1}$ if there exists a neighborhood $U\subset\C^{d-1}$ of $(c_0,a_0)$ such that the family $\{(c,a)\mapsto P_{c,a}^n(c_i)\}_{n\geq1}$ is normal on $U$. Otherwise, we say that $c_i$ is \emph{active} at $(c_0,a_0)$.
\end{definition}

It is known that the activity locus of $c_i$, i.e. the set of $(c_0,a_0)\in\C^{d-1}$ such that $c_i$ is active at $(c_0,a_0)$, coincides exactly with $\partial\mathcal{B}_i$ and that the bifurcation locus is exactly $\bigcup_i\partial\mathcal{B}_i$ (see e.g. \cite{lyubich,MSS,McMullen}). We let
$$T_i:=dd^cg_{c,a}(c_i)~.$$
Recall that the \emph{mass} of a closed positive $(1,1)$-current $T$ on $\C^{d-1}$ is given by
$$\|T\|:=\int_{\C^{d-1}}T\wedge\omega_{\textup{FS}}^{d-2}=\langle T,\omega_{\textup{FS}}^{d-2}\rangle~,$$
where $\omega_{\textup{FS}}$ stands for the Fubini-Study form on $\p^{d-1}(\C)$ normalized so that $\|\omega_{\textup{FS}}\|=1$ and that, if $T$ has finite mass, then it extends naturally as a closed positive $(1,1)$-current $\tilde T$ on $\p^{d-1}(\C)$ (see \cite{Demailly}). We also let $\|T\|_\Omega:=\langle T,\mathbf{1}_\Omega\omega_{\textup{FS}}^{d-2}\rangle$ for any open set $\Omega\subset\C^{d-2}$. One can prove the following (see \cite{DeMarco1,favredujardin}).

\begin{lemma}[Dujardin-Favre]
The support of $T_i$ is exactly $\partial \mathcal{B}_i$. Moreover, $T_i$ has mass $1$ and $T_i\wedge T_i=0$.\label{lm:DF}
\end{lemma}
On the other hand, the measure $\mu_{c,a}:=dd^c_z g_{c,a}(z)$ is the maximal entropy measure of $P_{c,a}$ and the Lyapounov exponent of $P_{c,a}$ with respect to $\mu_{c,a}$ is given by
$$L(c,a):=\int_\C\log|P_{c,a}'|\mu_{c,a}~.$$
A double integration by part gives
$$L(c,a)=\log d+\sum_{i=0}^{d-2}g_{c,a}(c_i).$$
In particular, the function $L:\C^{d-1}\to\R$ is plurisubharmonic and continuous and the $(1,1)$-current $dd^cL=\sum_iT_i$ is supported by the bifurcation locus.
\begin{definition}
The \emph{bifurcation current} is $T_\bif:=\sum_iT_i=dd^cL$.
\end{definition}

\subsection{The higher bifurcation currents and the bifurcation measure of $\mathcal{P}_d$}
Bassanelli and Berteloot \cite{BB1} introduce the \emph{higher bifurcation currents} and the \emph{bifurcation measure} on $\C^{d-1}$ (and in fact in a much more general context) by setting
$$T_\bif^k:=(dd^cL)^k \ \text{ and } \ \mu_\bif:=(dd^cL)^{d-1}~.$$
Dujardin and Favre \cite{favredujardin} and Dujardin \cite{dujardin2} study extensively the mesure $\mu_\bif$ in the present context. For our purpose, we first shall notice that Lemma \ref{lm:DF} implies that
\begin{eqnarray}
T_\bif^k=k!\sum_{0\leq i_1<\cdots<i_k\leq d-2}T_{i_1}\wedge \cdots \wedge T_{i_k}\label{eq:DeM}
\end{eqnarray}
is a positive closed $(k,k)$-current of finite mass. Let us set $$G(c,a):=\max_{0\leq i\leq d-2}\left(g_{c,a}(c_i)\right)~, \ \ (c,a)\in\C^{d-1}$$
and, for any $k$-tuple $I=(i_1,\ldots,i_k)$ with $0\leq i_1<cdots<i_k\leq d-2$ and $k\leq d-2$,
$$G_I(c,a):=\max_{1\leq j\leq k}\left(g_{c,a}(c_{i_j})\right)~, \ \ (c,a)\in\C^{d-1}~.$$
We shall use the following (see \cite[\S 6]{favredujardin}):
\begin{proposition}
Let $1\leq k\leq d-2$ and let $I=(i_1,\ldots,i_k)$ be a $k$-tuple with $0\leq i_1<\cdots<i_k\leq d-2$. Then $T_{i_1}\wedge\cdots\wedge T_{i_k}=(dd^cG_I)^k~.$ Moreover, $\mu_\bif=(d-1)!\cdot (dd^cG)^{d-1}$.
\label{propTk}
\end{proposition}

One of the crucial points of our proof relies on the following property of the measure $\mu_\bif$ (see \cite[Proposition 7 \& Corollary 11]{favredujardin}).

\begin{theorem}[Dujardin-Favre]\label{tm:mupp}
The support of $\mu_\bif$ coincides with the Shilov boundary $\partial_S\mathcal{C}_d\subset\partial \mathcal{C}_d$ of the connectedness locus. Moreover, there exists a Borel set $\mathcal{B}\subset\partial_S \mathcal{C}_d$ of full measure for the bifurcation measure $\mu_\bif$ and such that for all $(c,a)\in\mathcal{B}$,
\begin{itemize}
\item all cycles of $P_{c,a}$ are repelling,
\item the orbit of each critical points are dense in $\J_{c,a}$,
\item $\K_{c,a}=\J_{c,a}$ is locally connected and $\dim_H(\J_{c,a})<2$.
\end{itemize}
\end{theorem}

\subsection{Connectedness of the escape locus of a critical piont}
 We will need the next Lemma in the sequel.
\begin{lemma}
The open set $\{(c,a)\in\C^{d-1}\, ; \ g_{c,a}(c_j)>0\}$ is connected for $0\leq j\leq d-2$.\label{lm:connected}
\end{lemma}
\begin{proof}[Proof of Lemma \ref{lm:connected}]
If $p\in H_\infty\setminus H_j$, then $\C^{d-1}$ can be foliated by all the complex lines $(\ell_t)_{t\in A}$ of $\C^{d-1}$ with direction $p$, where $A$ is a $(d-2)$-dimensional complex plane which is transverse to the foliation. Let now $\ell$ be such a line. The choice of $p$ guarantees that $\ell\cap \{g_{c,a}(c_j)=0\}$ is a compact subset $\ell$. In particular, if the set $\ell\cap \{g_{c,a}(c_j)>0\}$ is not connected, it admits a bounded connected component $U$. By the maximum principle
$$\sup_U|P_{c,a}^n(c_j)|=\sup_{\partial U} |P_{c,a}^n(c_j)|~.$$
Since $\partial U$ is a compact subset of $\{g_{c,a}(c_j)=0\}\cap\ell$, the sequence $\{P_{c,a}^n(c_j)\}_{n\geq1}$ is uniformly bounded on $\partial U$, hence on $U$. This contradicts the fact that $U$ is a connected component of $\ell\cap \{g_{c,a}(c_j)>0\}$.

Now, if $(c,a),(c',a')\in\{g_{c,a}(c_j)>0\}$, the exists a ball $B\subset A$ such that $(c,a),(c',a')\in O:=\bigcup_{t\in B}\ell_t$. Since $\overline{B}$ is compact in $A$, there exists $R>0$ such that the set $\{g_{c,a}(c_j)=0\}\cap O$ is contained in $\B(0,R)$. Let now $t_0,t_1\in A$ be such that $(c,a)\in\ell_{t_0}$ and $(c',a')\in\ell_{t_1}$ and let $(c_0,a_0)\in \ell_{t_0}\setminus \B(0,R)\cap\ell_{t_0}$ and $(c_1,a_1)\in \ell_{t_1}\setminus \B(0,R)\cap\ell_{t_1}$. As $\ell_{t_0}\cap\{g_{c,a}(c_{j_{\tau(l)}})>0\}$ is a connected open subset of $\ell_{t_0}$, there exists a continuous path $\gamma_0:[0,1]\to\ell_{t_0}\cap\{g_{c,a}(c_j)>0\}$ with $\gamma_0(0)=(c,a)$ and $\gamma_0(1)=(c_0,a_0)$. One can find the same way a continuous path $\gamma_1:[0,1]\to\ell_{t_1}\cap\{g_{c,a}(c_j)>0\}$ with $\gamma_1(0)=(c_1,a_1)$ and $\gamma_1(1)=(c',a')$. Finally, the choice of $(c_0,a_0)$ and $(c_1,a_1)$ easily gives a continuous path $\gamma_3:[0,1]\to\{g_{c,a}(c_j)>0\}$ which  satisfies $\gamma_3(0)=(c_0,a_0)$ and $\gamma_3(1)=(c_1,a_1)$. The path $\gamma:=\gamma_1*\gamma_3*\gamma_2:[0,1]\to\{g_{c,a}(c_j)>0\}$ is continuous and satisfies $\gamma(0)=(c,a)$ and $\gamma(1)=(c',a')$, which ends the proof.
\end{proof}

\subsection{Horizontal currents and admissible wedge product}
We also need some known results concerning \emph{horizontal} currents. Let  $\Omega\subset \mathcal{X}$ be a connected open set of a complex manifold.
\begin{definition}
A closed positive $(1,1)$-current $T$ is \emph{horizontal} in $\Omega\times\D$ if the support of $T$ is an horizontal subset of $\Omega\times\D$, i.e. if there exists a compact set $K\Subset\D$ such that
$$\supp(T)\subset\Omega\times K~.$$
 We define similarly vertical currents.
\end{definition}

Following exactly the proof of \cite[Lemma 2.3]{DDS}, one gets the following.

\begin{lemma}\label{lm:horizontal}
Let $T$ be horizontal in $\Omega\times\D$. Let, for any $z\in\Omega$, $\mu_z:=T\wedge [\{z\}\times\D]$ be the slice of $T$ on the vertical slice $\{z\}\times\D$. Then the function 
$$u(z,w):=\int_{\{z\}\times\D}\log|w-s|\, d\mu_z(s)~$$
is 	a psh potential of $T$, i.e. $T=dd^cu$.
\end{lemma}

Assume now that $\Omega\subset\mathcal{M}$, where $\mathcal{M}$ is a complex manifold of dimension $n\geq2$.
Let $(T_\alpha)_{\alpha\in\mathcal{A}}$ be a measurable family of positive closed $(q,q)$-currents in $\Omega$ and let $\nu$ be a positive measure on $\mathcal{A}$ such that $\alpha\mapsto\|T_\alpha\|_\Omega$ is $\nu$-integrable. The \emph{direct integral} of $(T_\alpha)_{\alpha\in\mathcal{A}}$ is the current $T$ defined by
$$\langle T,\varphi\rangle:=\int_\mathcal{A}\langle T_\alpha,\varphi\rangle d\nu(\alpha)~,$$
for any $(n-q,n-q)$-test form $\varphi$. We denote $T$ by $T=\int_\mathcal{A}T_\alpha d\nu(\alpha)$.

Recall also that, if $T=dd^cu$ is a closed positive $(1,1)$-current and $S$ is a closed positive $(p,p)$-current with $p+1\leq n$, we say that the wedge product $T\wedge S$ is \emph{admissible} if $u\in L^1_\textup{loc}(\sigma_S)$, where $\sigma_S$ is the trace measure of $S$. It is classical that we then may define $T\wedge S:=dd^c(uS)$. Dujardin \cite[Lemma 2.8]{dujardin2} can be restated as follows:

\begin{lemma}\label{lm:admissible}
Let $T=\int_\mathcal{A}T_\alpha\, d\nu(\alpha)$ be a $(1,1)$-current as above and let $S$ be a closed positive $(p,p)$-current with $p+1\leq n$. Assume that the product $T\wedge S$ is admissible. Then, for $\nu$-almost every $\alpha$, $T_\alpha\wedge S$ is admissible and
$$T\wedge S=\int_\mathcal{A}\left(T_\alpha\wedge S\right)\, d\nu(\alpha)~.$$
\end{lemma}

\section{A comparison principle for plurisubharmonic functions}\label{sec:comp}

 We aim here at proving Theorem \ref{tmcompXav}. In the whole section, $\mathcal{X}$ stands for a $k$-dimensional complex manifold, $k\geq1$, for which there exists a smooth psh function $w$ on $\mathcal{X}$ and a strict analytic subset $\mathcal{Z}$ of $\mathcal{X}$ such that $(dd^cw)^k$ is a non-degenerate volume form on $\mathcal{X}\setminus\mathcal{Z}$. Let also $\Omega$ stand for a connected open subset of $\mathcal{X}$ with $\mathcal{C}^1-$smooth boundary. Let $\mathcal{PSH}(\Omega)$ stand for the set of all \emph{p.s.h} functions on $\Omega$ and let $\mathcal{PSH}^-(\Omega)$ be the set of non-positive $p.s.h$ functions on $\Omega$.

\subsection{Mass comparison for Monge-Amp\`ere measures}

\par We now assume in addition that $\overline{\Omega}$ is a compact subset of $\mathcal{X}$. We shall need the following lemma. Even though it looks very classical, we give a proof.

\begin{lemma}
Let $0\leq j\leq k$ and $u,v\in\mathcal{PSH}(\Omega)$ be such that $u=v$ on a neighborhood of $\partial\Omega$. Let $\omega$ be a smooth closed positive $(1,1)$-form. Assume that the measures $(dd^cu)^j\wedge\omega^{k-j}$ and $(dd^cv)^j\wedge\omega^{k-j}$ are well-defined and that $(dd^cu)^j\wedge\omega^{k-j}$ has finite mass. Then
\begin{eqnarray*}
\int_\Omega(dd^cv)^j\wedge\omega^{k-j}=\int_\Omega(dd^cu)^j\wedge\omega^{k-j}.
\end{eqnarray*}
\label{lmmasse}
\end{lemma}

\begin{proof}
For $j=0$, there is nothing to prove. We thus assume $j>0$. Let $\textup{d}$ be the distance induced by a Riemannian metric on  on $\mathcal{X}$. For $r>0$, we denote by $\Omega_r:=\{z\in\Omega / \textup{d}(z,\partial\Omega)>r\}$. Let $r>0$ be such that $u=v$ in a neighborhood of $\partial\Omega_r$ and let $\chi_r\in\mathcal{C}^\infty(\Omega)$ have compact support and be such that $0\leq\chi_r\leq1$ and $\chi_r=1$ on $\overline{\Omega_r}$ and $\chi_r=0$ on $\Omega\setminus\Omega_{r/2}$. Let $u_n$ be a decreasing sequence of smooth \emph{psh} functions on $\Omega$ converging to $u$ and $v_n$ be a decreasing sequence of smooth \emph{psh} functions on $\Omega$ converging to $v$. As $u=v$ in a neighborhood of $\partial\Omega$, for $n$ large enough, we may assume that $u_n=v_n$ in $\Omega\setminus\overline{\Omega_r}$. An integration by parts yields
\begin{eqnarray*}
\int_{\Omega}\chi_r(dd^cv_n)^j\wedge\omega^{k-j} & = & -\int_{\Omega}d\chi_r\wedge d^cv_n\wedge(dd^cv_n)^{j-1}\wedge\omega^{k-j}\\
 & = & -\int_{\Omega}d\chi_r\wedge d^cu_n\wedge(dd^cu_n)^{j-1}\wedge\omega^{k-j}=\int_{\Omega}\chi_r(dd^cu_n)^j\wedge\omega^{k-j},
\end{eqnarray*}
since $u_n=v_n$ on a neighborhood of $\supp(d\chi_r)\subset\Omega\setminus\overline{\Omega_r}$ for $n$ large enough. Letting $n$ tend to infinity, the above gives:
\begin{eqnarray*}
\int_{\Omega}\chi_r(dd^cv)^j\wedge\omega^{k-j}=\int_{\Omega}\chi_r(dd^cu)^j\wedge\omega^{k-j}.
\end{eqnarray*}
For $r'\leq r$, we can choose $\chi_{r'}\geq\chi_r$. As $r$ can be taken arbitrarily close to $0$, the monotonic convergence Theorem gives the wanted result.
\end{proof}

\subsection{Classical comparison principle}
\par We give here a local comparison theorem in the spirit of \cite[Corollary 2.3]{zeriahi}. It is one of the numerous generalizations of Bedford and Taylor classical comparison Theorem for Monge-Amp\`ere measures (see \cite{BedfordTaylor}).The difference with respect to Benelkourchi, Guedj and Zeriahi's work consists in the boundary condition, which is of different nature. The proof goes essentially the same way.

\begin{theorem}[Classical comparison principle]
Let $u,v\in\mathcal{PSH}^-(\Omega)$ be such that $(dd^cu)^k$ and $(dd^cv)^k$ are well-defined finite positive measures. Assume that $v\in\mathcal{C}(\Omega)$ and
\begin{center}
$\displaystyle\liminf_{\Omega\ni z\rightarrow z_0}(u(z)-v(z))\geq0$
\end{center}
for any $z_0\in\partial\Omega$. Then
\begin{eqnarray*}
\int_{\{u<v\}}(dd^cv)^k\leq\int_{\{u<v\}}(dd^cu)^k.
\end{eqnarray*}
\label{tmcomp}
\end{theorem}

\begin{proof}
Let $\epsilon>0$ and $\chi\pe \max(u+\epsilon,v)$. By assumption, the measure $(dd^c\chi)^k$ is well-defined and $\chi=u+\epsilon$ on a neighborhood of $\partial\Omega$. Lemma \ref{lmmasse}, thus gives
\begin{eqnarray*}
\int_\Omega(dd^c\chi)^k=\int_\Omega(dd^cu)^k.
\end{eqnarray*}
On the other hand, as $\chi$ is locally uniformly bounded and psh, we have
\begin{center}
$\mathbf{1}_{\{u+\epsilon<v\}}(dd^cv)^k=\mathbf{1}_{\{u+\epsilon<v\}}(dd^c\chi)^k$ and $\mathbf{1}_{\{u+\epsilon>v\}}(dd^cu)^k=\mathbf{1}_{\{u+\epsilon>v\}}(dd^c\chi)^k$.
\end{center}
Therefore, since $\{u+\epsilon\leq v\}\subset\{u<v\}$, we find:
\begin{eqnarray*}
\int_{\{u+\epsilon<v\}}(dd^cv)^k & = & \int_{\{u+\epsilon<v\}}(dd^c\chi)^k=\int_{\Omega}(dd^c\chi)^k-\int_{\{u+\epsilon\geq v\}}(dd^c\chi)^k\\
& = & \int_{\Omega}(dd^cu)^k-\int_{\{u+\epsilon> v\}}(dd^c\chi)^k-\int_{\{u+\epsilon=v\}}(dd^c\chi)^k\\
& \leq & \int_{\Omega}(dd^cu)^k-\int_{\{u+\epsilon> v\}}(dd^c\chi)^k=\int_{\Omega}(dd^cu)^k-\int_{\{u+\epsilon> v\}}(dd^cu)^k\\
& \leq & \int_{\{u+\epsilon\leq v\}}(dd^cu)^k \leq \int_{\{u<v\}}(dd^cu)^k.
\end{eqnarray*}
As $\mathbf{1}_{\{u+\epsilon<v\}}$ is an increasing sequence which converges pointwise to $\mathbf{1}_{\{u<v\}}$, by Lebesgue monotonic convergence Theorem, we conclude making $\epsilon\rightarrow0$.
\end{proof}

\par As in the classical case of locally uniformly bounded psh functions, we get as a consequence of Theorem \ref{tmcomp} the  following local domination principle. Notice that, up to now, we did not need the existence of $w$ as in Theorem~\ref{tmcompXav}. We shall now use this assumption.

\begin{corollary}[Classical domination principle]
Let $u,v\in\mathcal{PSH}^-(\Omega)$ be such that $(dd^cu)^k$ and $(dd^cu)^k$ are finite well-defined positive measures. Assume that $v\in\mathcal{C}(\Omega)$, that $u\geq v$ $(dd^cu)^k$-a.e. and that
\[\liminf_{\Omega\ni z\rightarrow z_0}(u(z)-v(z))\geq0~,\]
for all $z_0\in\partial\Omega$. Then $u\geq v$.
\label{cortmcomp}
\end{corollary}

\begin{proof}
We proceed by contradiction. Assume that the open set $\{u<v\}$ is non-empty. By our assumption on $\mathcal{X}$, there exists $w\in\mathcal{PSH}(\mathcal{X})\cap\mathcal{C}^\infty(\mathcal{X})$ such that $(dd^cw)^k$ is a non-degenerate volume form on $\mathcal{X}\setminus\mathcal{Z}$, where $\mathcal{Z}$ is an analytic subset of $\mathcal{X}$. As $\overline{\Omega}$ is a compact subset of $\mathcal{X}$, $w$ is bounded on $\Omega$ and, up to adding some negative constant to $w$, we may assume that $w\leq0$. For $\epsilon>0$, we set $v_{\epsilon}\pe v+\epsilon w$, then $v_\epsilon\leq v$ and $\{u<v_\epsilon\}\subset\{u<v\}$. If $\epsilon$ is small enough, the open set $\{u<v_\epsilon\}$ is also non-empty and
\begin{eqnarray*}
0<\epsilon^k\int_{\{u<v_\epsilon\}}(dd^cw)^k\leq\int_{\{u<v_\epsilon\}}(dd^cv_\epsilon)^k\leq\int_{\{u<v_\epsilon\}}(dd^cu)^k\leq\int_{\{u<v\}}(dd^cu)^k=0,
\end{eqnarray*}
which is the wanted contradiction.
\end{proof}

\subsection{Proof of Theorem \ref{tmcompXav}}

First, we prove that $u=v$ on $\partial K$. Let $z_0\in\partial K$. As $v$ is continuous and $u$ is usc,
\[v(z_0)=\limsup_{K\not\ni z\to z_0}v(z)=\limsup_{K\not\ni z\to z_0}u(z)\leq\limsup_{z\to z_0}u(z)\leq \limsup_{z\to z_0}v(z)=v(z_0)~,
\]
and thus $u=v$ on $\partial K$. Let now $U$ be a connected component of $\mathring{K}$ and let us set
\[\rho(z) = \begin{cases} u(z) &\text{if }z\in  U\\ 
v(z) &\text{if }z\in\Omega\setminus U.
\end{cases}
\]
The function $\rho$ is then psh on $\Omega \setminus\partial U$ and usc on $\Omega$. Moreover, if $z_0\in\partial U$, $\rho$ satisfies the submean inequality at $z_0$ in any non-constant holomorphic disk $\sigma:\D\to\Omega$ with $\sigma(0)=z_0$. Indeed, if $r>0$ is small, then
\[\rho(z_0)=u(z_0)=u\circ\sigma(0)\leq\frac1{2\pi}\int_0^{2\pi}u\circ\sigma(re^{i\theta})d\theta\leq\frac1{2\pi}\int_0^{2\pi}\rho\circ\sigma(re^{i\theta})d\theta~,
\]
where the last inequality comes from the fact that, by definition of $\rho$, we have $u\leq \rho$. Hence, $\rho$ is psh on $\Omega$. By \cite[Prop. 4.1, p. 150]{Demailly}, since $\rho=v$ outside of a compact subset of $\Omega$, the measure $(dd^c\rho)^k$ is well-defined. According to Lemma \ref{lmmasse}, it comes
\[(dd^c\rho)^k(\Omega)=(dd^cv)^k(\Omega)~.\]
Moreover, by definition of $\rho$, one has $(dd^c\rho)^k=(dd^cv)^k$ on $\Omega\setminus \overline{U}$. Thus,
\begin{eqnarray*}
(dd^c\rho)^k(\overline{U}) & = & (dd^cw)^k(\Omega)-(dd^c\rho)^k(\Omega\setminus\overline{U})\\
& = & (dd^cv)^k(\Omega)-(dd^cv)^k(\Omega\setminus\overline{U})\\
& = & (dd^cv)^k(\overline{U})=0~,
\end{eqnarray*}
which gives $(dd^c\rho)^k=(dd^cv)^k$ as measures on $\Omega$. Since $\rho\geq v$ on $\supp((dd^cv)^k)$, this in particular implies that $\rho\geq v$, $(dd^c\rho)^k$-a.e in $\Omega$.

~

\par Let $W\Subset \Omega$ be an open set with smooth boundary such that $K\Subset W$ and let $M:=\sup_{W}v\in\R$. Set now $\rho_1:=\rho-M$ and $v_1:=v-M$. To conclude, we want to apply Corollary~\ref{cortmcomp} to $\rho_1$ and $v_1$ on $W$. From the above discussion, we have $v_1\in\mathcal{C}(\Omega)$, $\rho_1,v_1\in\mathcal{PSH}^-(W)$, $(dd^c\rho_1)^k=(dd^cv_1)^k$ is a finite well-defined positive measure, and $\rho_1=v_1$ on $\supp((dd^c\rho_1)^k)$ and on a neighborhood of $\partial W$. According to Corollary \ref{cortmcomp}, we then have $\rho_1\geq v_1$ on $W$. In particular, $u=\rho=v$ on $U$. As this remains valid for any connected component $U$ of $\mathring{K}$, we have proved that $u=v$ on $\mathring{K}$, which ends the proof.

\section{Structure of some slices of the bifurcation currents}\label{sec:structure}

Pick $d\geq3$ once and for all. For any $1\leq q\leq d-2$, we set $\ell:=d-q$,
$$\Sigma_{\ell}:=\{1,\ldots,\ell\}^\mathbb{N}~,$$
and let $\sigma_{\ell}:\Sigma_{\ell}\to\Sigma_{\ell}$ be the full shift on $\ell$ symbols, i.e. $\sigma_{\ell}(\epsilon_0\epsilon_1\cdots)=\epsilon_1\epsilon_2\cdots$ for all $\epsilon=\epsilon_0\epsilon_1\cdots\in\Sigma_\ell$.

\par When $\ud:=(d_1,\ldots,d_{\ell})\in(\mathbb{N}^*)^\ell$ satisfies $d=\sum_id_i$, we also let $\nu_{\ud}$ be the probability measure on $\Sigma_{\ell}$, which is invariant by $\sigma_{\ell}$, and giving mass $(d_{\epsilon_0}\cdots d_{\epsilon_{n-1}})/d^n$ to the cylinder of sequences starting with $\epsilon_0,\ldots,\epsilon_{n-1}$. 

\begin{definition}
The measure $\nu_{\ud}$ is called the $\ud$-\emph{measure} on $\Sigma_{\ell}$.
\end{definition}
\noindent Let us remark that by definition, the $\ud$-measure $\nu_{\ud}$ does not give mass to points.

~

For the whole section, we let $I=(i_1,\ldots,i_k)$ be a $k$-tuple with $0\leq i_1<\ldots<i_k\leq d-2$ and we let $I^c$ be the unique $(d-1-k)$-tuple satisfying $I\cup I^c=\{0,1,\ldots,d-2\}$. We may write $I^c=(j_1,\ldots,j_{d-1-k})$.  For any $\tau\in\mathfrak{S}_{d-1-k}$, we let 
$$U_{I,\tau}:=\{G_I(c,a)<g_{c,a}(c_{j_{\tau(1)}})\}\cap\bigcap_{l=1}^{d-k-2}\{g_{c,a}(c_{j_{\tau(l)}})<g_{c,a}(c_{j_{\tau(l+1)}})\}~.$$
This section is devoted to the proof the following.

\begin{theorem}\label{cor:TI}
For any $(c,a)\in U_{I,\tau}\cap \{G_I=0\}$, there exists an analytic set $\mathcal{X}_0\subset U_{I,\tau}$, a complex manifold $\mathcal{X}$ and a finite holomorphic map $\pi:\mathcal{X}\to\mathcal{X}_0$ such that:
\begin{enumerate}
\item $(c,a)\in\mathcal{X}_0$, $\mathcal{X}$ has dimension $k$ and $\{G_I\circ\pi=0\}\Subset\mathcal{X}$,
\item $(dd^cL\circ\pi)^k$ is a finite measure on $\mathcal{X}$ supported by $\partial\left(\{G_I\circ\pi=0\}\right)$,
\item there exists $2\leq \ell \leq d-k$ and $\ud=(d_1,\ldots,d_\ell)\in(\mathbb{N}^*)^\ell$ with $d=\sum_id_i$ and such that, for any $\epsilon\in\Sigma_\ell$, there exists $k$ closed positive $(1,1)$-current $T_{\epsilon,1},\ldots,T_{\epsilon,k}$ on $\mathcal{X}$ with $L_{\textup{loc}}^\infty$ potentials such that the wedge product $T_{\epsilon_1,1}\wedge\cdots\wedge T_{\epsilon_k,k}$ is admissible for $\nu_{\ud}^{\otimes k}$-a.e. $\epsilon=(\epsilon_1,\ldots,\epsilon_k)\in\Sigma_{\ell}^k$ and, as measures on $\mathcal{X}$,
$$(dd^cL\circ\pi)^k=k!\int_{\Sigma_{\ell}^k}T_{\epsilon_1,1}\wedge \cdots \wedge T_{\epsilon_k,k}\, d\nu_{\ud}^{\otimes k}(\epsilon)~.$$
\item for any connected component $\mathcal{U}$ of the interior of $\{G_I\circ\pi=0\}$,
\begin{center}
$(dd^cL\circ\pi)^k(\partial\mathcal{U})=0$.
\end{center}
\end{enumerate}
\end{theorem}
%\begin{theorem}\label{tm:TI}
%$T_{i_1}\wedge \cdots \wedge T_{i_k}$ is uniformly laminar in $U_{I,\sigma}$.
%\end{theorem}
\subsection{Preliminaries to Section \ref{sec:structure}}
\subsubsection{Further on the bifurcation current of a critical point}
For the material of this paragraph, we refer to \cite{favredujardin,dujardin2}. Let $\mathcal{X}$ be any complex manifold and let $(P_\lambda)_{\lambda\in\mathcal{X}}$ be any holomorphic family of degree $d$ polynomials. One can define a fibered dynamical system $\widehat P$ acting on $\widehat{\mathcal{X}}:=\mathcal{X}\times\C$ as follows
\begin{eqnarray*}
\widehat P:\mathcal{X}\times\C & \longrightarrow & \mathcal{X}\times\C\\
(\lambda,z) & \longmapsto & (\lambda,P_{\lambda}(z))~.
\end{eqnarray*}
The sequence $d^{-n}\log^+|(\widehat P)^n|$ converges uniformly locally on $\mathcal{X}\times \C$ to the continuous psh function $(\lambda,z)\mapsto g_{\lambda}(z)$, where $g_\lambda$ is the Green function of $P_\lambda$. Let us set
$$\widehat T_{\mathcal{X}}:=dd^c_{\lambda,z}g_\lambda(z)$$
and let $p_1:\mathcal{X}\times\C\to\mathcal{X}$ and $p_2:\mathcal{X}\times\C\to\C$ be the respective natural projections. Assume in addition that $(P_\lambda)_{\lambda\in\mathcal{X}}$ is endowed with $d-1$ marked critical points, i.e. that there exists holomorphic functions $c_1,\ldots,c_{d-1}:\mathcal{X}\to\C$ with $C(P_\lambda)=\{c_1(\lambda),\ldots,c_{d-1}(\lambda)\}$. In this setting, one can easily see that
$$T_{i}=dd^c\left(g_\lambda(c_i(\lambda)\right)=(p_1)_*\left(\widehat T_{\mathcal{X}}\wedge [\mathcal{C}_i]\right)~,$$
where $\mathcal{C}_i=\{(\lambda,c_i(\lambda))\}$ is the graph of the map $\lambda\mapsto c_i(\lambda)$.

~

\subsubsection{B\"ottcher coordinate of a $P_{c,a}$ at infinity}

Recall that the \emph{B\"ottcher coordinate} of $P_{c,a}$ at infinity is the biholomorphic map
\[\psi_{c,a}:W_{c,a}:=\{z\in\C \, | \ g_{c,a}(z)>G(c,a)\}\to\C\setminus\overline{\D(0,e^{G(c,a)})}\]
which satisfies $\psi_{c,a}(z)=z+O(1)$ at infinity and that conjugates $P_{c,a}$ to $z^d$:
$$\psi_{c,a}\circ P_{c,a}(z)=(\psi_{c,a}(z))^d, \ z\in W_{c,a}~.$$
Notice that $\psi_{c,a}$ depends holomorphically on $(c,a)$ and that, for $z\in W_{c,a}$, one can prove that $g_{c,a}(z)=\log|\psi_{c,a}(z)|$ (see e.g. \cite{Milnor4}).

\subsection{The maximal entropy measure $\mu_{c,a}$ for $(c,a)\in \C^{d-1}\setminus\mathcal{C}_d$}\label{sec:mu}

First, we want to prove that, when exactly $d-k-1$ critical points of $P_{c,a}$ escape, the maximal entropy measure $\mu_{c,a}$ of $P_{c,a}$ enjoys good decomposition properties with respect to some $\ud$-measure $\nu_{\ud}$ for some $2\leq\ell\leq d-k$. Namely, we prove the following.

\begin{proposition}\label{prop:mu}
Let $(c,a)\in\C^{d-1}\setminus\mathcal{C}_d$. Assume that $d-k-1$ critical points (counted with mulitplicity) of $P_{c,a}$ escape under iteration. Then, there exists $k\leq q\leq d-2$ such that one can decompose $\K_{c,a}$ as a disjoint union of (possibly non-connected) compact sets
$$\K_{c,a}=\bigcup_{\epsilon\in\Sigma_{\ell}}\K_\epsilon$$
where $\ell=d-q$. Moreover, the following holds
\begin{enumerate}
\item there exists $\ud\in(\mathbb{N}^*)^\ell$ with $d_1+\cdots+d_\ell=d$ and for any $\epsilon\in\Sigma_{\ell}$, there exists a probability measure $\mu_\epsilon$ supported by $\K_\epsilon$ such that $\mu_\epsilon=\Delta g_\epsilon$, where $g_\epsilon$ is subharmonic and locally bounded and, as probability measures on $\C$,
\[\mu_{c,a}=\int_{\Sigma_{\ell}}\mu_\epsilon d\nu_{\ud}(\epsilon)~,\]
\item for any $\epsilon\in\Sigma_{d-q}$, one has $\mu_{c,a}(\K_\epsilon)=0$.
\end{enumerate}
\end{proposition}

\begin{proof}
We follow closely the strategy of the proof of \cite[Theorem 3.12]{dujardin2} and adapt it to our situation. According to \cite[Theorem 9.3]{Milnor4}, the real curve 
\[\{z\in\C \, | \ g_{c,a}(z)=G(c,a)>0\}\] contains at least one critical point of $P_{c,a}$. Let us define a topological disk $U_0$ by setting $U_0:=\{z\in\C \, | \ g_{c,a}(z)<d\cdot G(c,a)\}$ and $U_1:=P_{c,a}^{-1}(U_0)$.
 \begin{lemma}
Any component of $U_1$ is a topological disk and $U_{1}\Subset U_0$.\label{lm:BuffHubbard1}
\end{lemma}
 We postpone the proof to the end of the subsection. As explained in the proof of \cite[Theorem 9.5]{Milnor4}, one can show that $U_1$ has at least 2 distinct connected components. We thus can find disjoint open set $V_1,\ldots,V_N$ so that $U_1=V_1\cup\cdots\cup V_N$. Let us set $P_i:=P_{c,a}|_{V_i}: V_i \to U_0$ is a ramified covering map of degree $d_i\geq 1$.
 
\begin{claim}
Let $q\geq k$ be the number of critical points of $P_{c,a}$ lying in $U_1$, counted with multiplicity. Then $U_1$ has $\ell:=d-q$ distinct connected components and
\begin{center}
$d=d_1+\cdots+d_{\ell}$.
\end{center}
\end{claim}

Let us continue the proof of Propostion \ref{prop:mu}. For any $\epsilon\in\Sigma_{d-q}$, we set
$$\K_\epsilon:=\{z\in U_0 \, ; \, P_{c,a}^{m}(z)\in V_{\epsilon_m}, \, m\geq0\}=\bigcap_{n\geq0}P_{\epsilon_0}^{-1}\left(\cdots\left( P_{\epsilon_n}^{-1}(\overline{U_0})\right)\right)~.$$
Beware that the set $\K_{c,a}$ has uncountably many connected components and that, for any given $\epsilon\in\Sigma_{d-q}$, the compact set $\K_\epsilon$ is not necessarily connected. In fact, whenever $C(P_{c,a})\cap U_1\not\subset\K_{c,a}$, there must exist non-connected $\K_\epsilon$. Clearly,
$$\K_{c,a}=\bigcup_{\epsilon\in\Sigma_{\ell}}\K_\epsilon$$
and this decomposition naturally gives a continuous surjective map
$$h_{c,a}:\K_{c,a}\longrightarrow\Sigma_{\ell}$$
satisfying $h_{c,a}(z)=\epsilon$ iff $z\in\K_\epsilon$. The map $h_{c,a}$ then semi-conjugates $P_{c,a}$ on $\K_{c,a}$ to $\sigma_{\ell}$ on $\Sigma_{\ell}$, i.e.  satisfies $h_{c,a}\circ P_{c,a}=\sigma_{\ell}\circ h_{c,a}$ on $\K_{c,a}$.

~

Proceeding as in \cite{dujardin2}, one gets the following: for any $z\in U_0\setminus\K_{c,a}$, one can rewrite $d^{-n}(P_{c,a}^n)^*\delta_z$ as follows
\begin{eqnarray}
\frac{1}{d^n}(P_{c,a}^n)^*\delta_z & = & \frac{1}{d^n}\sum_{\epsilon_i\in\{1,\ldots,\ell\}, i\leq n}P_{\epsilon_0}^*\cdots P_{\epsilon_{n-1}}^*\delta_z\nonumber\\
& = & \sum_{\epsilon_i\in\{1,\ldots,\ell\}, i\leq n}\frac{d_{\epsilon_0}\cdots d_{\epsilon_{n-1}}}{d^n}\left[\frac{1}{d_{\epsilon_0}\cdots d_{\epsilon_{n-1}}}P_{\epsilon_0}^*\cdots P_{\epsilon_{n-1}}^*\delta_z\right]~.\label{eq:romain}
\end{eqnarray}
When $n\to\infty$, the following convergence holds independently of $z$,
\begin{eqnarray*}
\frac{1}{d_{\epsilon_0}\cdots d_{\epsilon_{n-1}}}P_{\epsilon_0}^*\cdots P_{\epsilon_{n-1}}^*\delta_z\longrightarrow_{n\to\infty}\mu_\epsilon~,
\end{eqnarray*}
where the measure $\mu_\epsilon$ is a probability measure supported by $\partial\K_\epsilon$. The measure
$\mu_\epsilon$ is the analogue of the Brolin measure for the sequence $(P_{\epsilon_i})_{i\geq0}$. In particular, one can write $\mu_\epsilon=\Delta g_\epsilon$, where $g_{\epsilon}$ is a locally bounded subharmonic function on $\C$.

As $d^{-n}(P_{c,a}^n)^*\delta_z$ converges to $\mu_{c,a}$ of $P_{c,a}$, making $n\to\infty$ in \eqref{eq:romain}, one finds
\begin{center}
$\mu_{c,a}=\displaystyle\int_{\Sigma_{\ell}}\mu_\epsilon d\nu_{\ud}(\epsilon)~.$
\end{center}
Let now $\epsilon\in\Sigma_{\ell}$. By the above, as $\nu_{\ud}$ does not give mass to points,
$$\mu_{c,a}(\K_{\epsilon})=\nu_{\ud}(\{\epsilon\})=0~,$$
which ends the proof.
\end{proof}

\begin{proof}[Proof of Lemma \ref{lm:BuffHubbard1}]
One first sees that
\begin{eqnarray*}
U_1 & = & \{z\in\C \, | \ \exists x \in U_0 \ \text{s.t.} \ P_{c,a}(z)=x\}\\
& = & \{z\in\C \, | \ g_{c,a}(P_{c,a}(z))<d\cdot G(c,a)\}\\
& = &  \{z\in\C \, | \ g_{c,a}(z)<G(c,a)\}\Subset \{z\in\C \, | \ g_{c,a}(z)<d\cdot G(c,a)\}=U_0~,
\end{eqnarray*}
as $g_{c,a}$ is the Green function of the compact set $\K_{c,a}$. Assume that some connected component $W$ of $U_1$ is not homeomorphic to a disk. Let $D$ be the unique unbounded component of $\C\setminus W$. By assumption, $\C\setminus W$ has a bounded component $O$ and $P(O)\subset D$. Hence $W$ contains a pole of $P_{c,a}$. This is impossible since, as $P_{c,a}$ is a polynomial, it has no poles in $\C$.
\end{proof}

\begin{proof}[Proof of the Claim]
First, the map $P_{c,a}:\p^1\setminus \overline{U_1}\to\p^1\setminus \overline{U_0}$ is a branched covering of degree $d$ and $\chi(\p^1\setminus \overline{U_0})=1$ and $\chi(\p^1\setminus \overline{U_1})=2-N$. Let $q\geq k$ be such that $d-q-1$ critical points of $P_{c,a}$ belong to $U_0\setminus\overline{U_1}$. As $\infty$ is a critical points of multiplicity $d-1$ of $P_{c,a}$, by Riemann-Hurwitz, one has
$$d\cdot 1=2-N+(d-q-1)+(d-1)=2d-q-N~,$$
which leads to $N=d-q$.

For any $1\leq j\leq d-q$, the map $P_i:V_i\to U_0$ is a branched covering. As $V_i$ is a topological disk, one has $\chi(V_i)=\chi(U_0)=1$ and the Riemann-Hurwitz formula gives
$$d_i=\deg(P_i)=r_i+1~,$$
where $r_i$ is the number of critical points of $P_{c,a}$ contained in $V_i$, counted with multiplicities. Making the sum over $i$, we find
$$\sum_{i=1}^{d-q}d_i=\sum_{i=1}^{d-q}(r_i+1)=d-q+\sum_{i=1}^{d-q}r_i=d~,$$
since $\sum_ir_i$ is the number of critical points contained in $U_1$, i.e. $\sum_ir_i=q$.
\end{proof}

\subsection{Decomposition of bifurcation currents in specific families}
 We are now in position to prove Theorem \ref{cor:TI}. We follow closely the strategy of the proof of \cite[Theorem 3.12]{dujardin2}. Pick a $k$-tuple $I=(i_1,\ldots,i_k)$, with $0\leq i_1<\cdots<i_k\leq d-2$ and let $\tau\in\mathfrak{S}_{d-1-k}$ and $(c_0,a_0)\in\{G_I=0\}\cap U_{I,\tau}$. Denote by $I^c$ the $(d-1-k)$-tuple such that $I\cup I^c=\{0,\ldots,d-2\}$. 
 First, we construct the analytic set $\mathcal{X}_0$ and define $\pi:\mathcal{X}\to\mathcal{X}_0$ as a desingularization. Properties $1$ and $2$ will easily follow from the construction. In a second time, we use Proposition~\ref{prop:mu} to show that the fibered Green current $\widehat{T}_{\mathcal{X}}$ is very close to laminar. Finally, we prove that properties $3$ and $4$ also hold on $\mathcal{X}$.

\subsubsection{First step: contruction of $\mathcal{X}_0$ and $\mathcal{X}$} Our first aim in the present subsection is to build $\mathcal{X}$. Namely, we prove

\begin{lemma}\label{lm:X}
For any $(c,a)\in U_{I,\tau}\cap \{G_I=0\}$, there exists an analytic set $\mathcal{X}_0\subset U_{I,\tau}$, a complex manifold $\mathcal{X}$ and a finite proper holomorphic map $\pi:\mathcal{X}\to\mathcal{X}_0$ such that:
\begin{enumerate}
\item $(c,a)\in\mathcal{X}_0$, $\mathcal{X}$ has dimension $k$ and $\{G_I\circ\pi=0\}\Subset\mathcal{X}$,
\item $(dd^cL\circ\pi)^k=k!\cdot (dd^cG_I\circ\pi)^k$ is a finite measure on $\mathcal{X}$ supported by $\partial\left(\{G_I\circ\pi=0\}\right)$.
\end{enumerate}
\end{lemma}

\begin{proof}
 As $g_{c_0,a_0}(c_{j,0})>0$ for any $j\in I^c$, there exists $k_j\geq1$ such that $g_{c_0,a_0}(P_{c_0,a_0}^{k_j}(c_{j,0}))=d^{k_j}g_{c_0,a_0}(c_{j,0})>G(c_0,a_0)$. Let us set
$$\mathcal{X}_1:=\bigcap_{j\in I^c}\{(c,a)\in U_{I,\tau} \, | \ \psi_{c,a}(P_{c,a}^{k_j}(c_j))=\psi_{c_0,a_0}(P_{c_0,a_0}^{k_j}(c_{j,0}))\}~.$$
Then $\mathcal{X}_1$ is an analytic variety of dimension at least $k$. Up to taking an irreducible component of $\mathcal{X}_1$, we may assume that it is irreducible. Moreover, it is contained in
$$\mathcal{Y}:=\bigcap_{j\in I^c}\{(c,a)\in U_{I,\tau} \, | \ g_{c,a}(c_j)=g_{c_0,a_0}(c_{j,0})\}~.$$
The boundary of $\mathcal{Y}$ consists in parameters $(c,a)$ for which $G_I(c,a)=g_{c_0,a_0}(c_{j_{\tau(1)},0})>0$. In particular, $\partial\mathcal{X}_1$ consists in parameters for which $G_I(c,a)= g_{c,a}(c_{j_{\tau(1)}})>0$, hence
\begin{enumerate}
\item $\partial\mathcal{X}_1\subset\partial U_{I,\tau}$ and
\item $\{G_I=0\}\cap\mathcal{X}_1\Subset \mathcal{X}_1$.
\end{enumerate}
Let now $q\geq k$ be the integer given by Proposition \ref{prop:mu} at the parameter $(c_0,a_0)$ and let
$$\mathcal{X}_0:=\mathcal{X}_1\cap \{d\cdot G_I(c,a)<G(c_0,a_0)\}~.$$
Let finally $\pi:\mathcal{X}\to\mathcal{X}_0$ be a desingularization of $\mathcal{X}_0$. We denote by $P_\lambda$ the polynomial $P_{c,a}$ if $(c,a)=\pi(\lambda)$. Let also $c_i(\lambda):=c_i\circ\pi(\lambda)$. We also let $\lambda_0\in\mathcal{X}$ be such that $\pi(\lambda_0)=(c_0,a_0)$. Let us remark that $\mathcal{X}$ still sasitsfies properties $1$ and $2$ aforementioned and that $(dd^cG_I\circ\pi)^k$ is supported by the compact set $\partial\{G_I\circ\pi=0\}$. Let $K\Subset\mathcal{X}$ be a compact subset with $\{G_I\circ\pi=0\}\Subset \mathring{K}$. By the Chern-Levine-Nirenberg inequalities, there exists a constant $C>0$ such that
$$\|(dd^cG_I\circ\pi)^k\|\leq C\cdot \|G_I\circ\pi\|_{L^\infty(K)}^k<+\infty~,$$
According to \eqref{eq:DeM} and Proposition \ref{propTk}, since $\supp(dd^cg_{c,a}(c_j))\subset\{g_{c,a}(c_j)=0\}$, one has
$$(dd^cL)^k=k!\cdot(dd^cG_I)^k\text{ on }U_{I,\tau}~,$$
which concludes the proof.
\end{proof}

To prove Theorem~\ref{cor:TI}, we can apply Lemma~\ref{lm:X} and it is just left to prove that $(dd^cG_I\circ\pi)^k$ satisfies the assertions $3$ and $4$ of the Theorem. This is a consequence of the two next paragraphs.

\subsubsection{Second step: Decomposition of the current $\widehat{T}_{\mathcal{X}}$} 
We now want prove the following.

\begin{proposition}\label{prop:laminar}
Pick $(c,a)\in U_{I,\tau}\cap \{G_I=0\}$. Let $\mathcal{X}$ be given by Lemma~\ref{lm:X} and $\nu_{\ud}$ and $\ell$ be given by Proposition~\ref{prop:mu}. For any $\epsilon\in\Sigma_\ell$, there exists a closed positive $(1,1)$-current $\widehat{T}_\epsilon$ on $\mathcal{X}\times\C$ such that in the weak sense of currents on $\mathcal{X}\times\C$,
 $$\widehat{T}_{\mathcal{X}}=\int_{\Sigma_{\ell}}\widehat{T}_\epsilon\, d\nu_{\ud}(\epsilon)~.$$
\end{proposition}

\begin{proof}
To begin, remark that the labelling $P_\lambda^{-1}(U_0)=V_1\cup\cdots\cup V_{\ell}$ introduced in the proof of Proposition \ref{prop:mu} does not depend on any choice. Moreover, according to the proof of Proposition \ref{prop:mu} and to the definition of $\mathcal{X}_0$, this decomposition persists in $\mathcal{X}$ and depends continously on the parameter $\lambda$.  We thus can define
$$\widehat P_i:\mathcal{X}\times V_i\longrightarrow\mathcal{X}\times U_0$$
by setting $\widehat P_i(\lambda,z)=(\lambda,P_{i,\lambda}(z))$. Let us also set $s(\lambda):=c_{j_{\tau(1)}}(\lambda)$. Let $R>0$ be big enough so that $U_{0,\lambda}\subset\D(0,R/2)$ for any $\lambda\in\mathcal{X}$. Such an $R$ exists by construction of $\mathcal{X}$ (Take for example $R=d^2G(c_0,a_0)$). Let $\ell\geq1$, $\ud\in(\mathbb{N}^*)^\ell$ and $\nu_{\ud}$ be given by Proposition \ref{prop:mu}. As we have seen in section \ref{sec:mu}, for any $\epsilon\in\Sigma_{\ell}$ and any $\lambda\in\mathcal{X}$, the sequence
\begin{eqnarray*}
\frac{1}{d_{\epsilon_0}\cdots d_{\epsilon_{n-1}}} P_{\epsilon_0,\lambda}^*\cdots P_{\epsilon_{n-1},\lambda}^*\delta_{s(\lambda)}=dd^c_z\left(\frac{1}{d_{\epsilon_0}\cdots d_{\epsilon_{n-1}}}\log\left|P_{\epsilon_{n-1},\lambda}\circ \cdots\circ P_{\epsilon_{0},\lambda}(z)-s(\lambda)\right|\right)
\end{eqnarray*}
converges to a measure $\mu_{\epsilon,\lambda}$ which has a $L^\infty_{\textup{loc}}$ logarithmic potential $g_{\epsilon,\lambda}$.

~

Let now $\Gamma_s$ be the graph of $s$, $\Gamma_s:=\{(\lambda,s(\lambda))\, ; \ \lambda\in\mathcal{X}\}$. We can write 
\begin{eqnarray}
\frac1{d^n}(\widehat P^*)^n[\Gamma_s]=\frac{1}{d^n}\sum_{\epsilon_i\in\{1,\ldots,\ell\}, \, i\leq n-1}\widehat P_{\epsilon_0}^*\cdots \widehat P_{\epsilon_n}^*[\Gamma_s]~.\label{eq:decomposition}
\end{eqnarray}
For $n\geq0$, one also can set
\begin{eqnarray*}
\widehat T_{\epsilon,n} & := & \frac{1}{d_{\epsilon_0}\cdots d_{\epsilon_{n-1}}}\widehat P_{\epsilon_0}^*\cdots\widehat P_{\epsilon_{n-1}}^*[\Gamma_{s}]\\
& = & dd^c_{\lambda,z}\left(\frac{1}{d_{\epsilon_0}\cdots d_{\epsilon_{n-1}}}\log\left|P_{\epsilon_{n-1},\lambda}\circ\cdots\circ P_{\epsilon_0,\lambda}(z)-s(\lambda)\right|\right)=dd^c_{\lambda,z}u_{\epsilon,n}~,
\end{eqnarray*}
where we have set
$$u_{\epsilon,n}(\lambda,z):=\frac{1}{d_{\epsilon_0}\cdots d_{\epsilon_{n-1}}}\log\left|P_{\epsilon_{n-1},\lambda}\circ\cdots\circ P_{\epsilon_0,\lambda}(z)-s(\lambda)\right|~.$$
It is obvious that the sequence $(u_{\epsilon,n})_{n\geq1}$ is locally uniformly bounded from above. According to Proposition \ref{prop:mu}, for any $\lambda\in\mathcal{X}$, the functions $u_{\epsilon,n}|_{\{\lambda\}\times\D(0,R)}$ converges in $L^1_\textup{loc}$ to a subharmonic function $\not\equiv-\infty$. Hence there exists a subsequence $(u_{\epsilon,n_k})$ which converges in $L^1_\textup{loc}(\mathcal{X}\times\D(0,R))$ to a psh function $u_{\epsilon,\infty}$. Let us remark that $\widehat T_{\epsilon,n_k}$ are all horizontal currents with supports contained in $\mathcal{X}\times\D(0,R/2)$. Making $k\to\infty$, we see that the current $\widehat T_{\epsilon,\infty}:=dd^cu_{\epsilon,\infty}$ is horizontal. According to Lemma \ref{lm:horizontal}, one can write
\begin{eqnarray}
u_{\epsilon,\infty}(\lambda,z)=\int_{\D(0,R)}\log|z-t|\, d\mu_{\epsilon,\lambda}(t)+h(\lambda,z)=g_{\epsilon,\lambda}(z)+h(\lambda,z)~,\label{eq:slice}
\end{eqnarray}
where $h$ is pluriharmonic on $\mathcal{X}\times \D(0,R)$ and $g_{\epsilon,\lambda}(z)$ is the logarithmic potential of $\mu_{\epsilon,\lambda}$.

In particular, the function $(\lambda,z)\mapsto g_{\epsilon,\lambda}(z)$ is psh on $\mathcal{X}\times \D(0,R)$ and the sequence $\widehat T_{\epsilon,n}$ converges in the weak sense of currents to $\widehat T_\epsilon:=dd^c_{\lambda,z}g_{\epsilon,\lambda}(z)$.

~

Recall that $\widehat{T}_{\mathcal{X}}=dd^c_{\lambda,z}g_\lambda(z)$. Again, we follow the argument of Dujardin \cite{dujardin2}: As $g_\lambda(s(\lambda))>0$, $s(\lambda)$ escapes under iteration and the sequence $\frac1{d^n}(\widehat P^*)^n[\Gamma_s]$ converges to $\widehat{T}_{\mathcal{X}}$ as $n\to\infty$. The decomposition \eqref{eq:decomposition} then guarantees that $\widehat{T}_{\mathcal{X}}=\int_{\Sigma_{\ell}}\widehat{T}_\epsilon\, d\nu_{\ud}(\epsilon)$.
\end{proof}

\subsubsection{Third step: Decomposition of the bifurcation currents of $\mathcal{X}$} 

The aim here is to prove that the bifurcation currents associated with critical points and the bifurcation measure are close to be laminar in $\mathcal{X}$. Our precise result can be stated as follows.

\begin{theorem}\label{tm:laminar}
Pick $(c,a)\in U_{I,\tau}\cap \{G_I=0\}$. Let $\mathcal{X}$ be given by Lemma~\ref{lm:X} and $\nu_{\ud}$ and $\ell$ be given by Proposition~\ref{prop:mu}. Write $I=(i_1,\ldots,i_k)$ . Then, for any $1\leq j\leq k$, and any $\epsilon\in\Sigma_\ell$, there exists a closed positive $(1,1)$-current $T_{\epsilon,j}$ such that
\[dd^cg_\lambda(c_{i_j}(\lambda))=\int_{\Sigma_\ell}T_{\epsilon,j}\, d\nu_{\ud}(\epsilon)~.\]
\end{theorem}

\begin{proof}
Let $p_1:\mathcal{X}\times\C\to\mathcal{X}$ and $p_2:\mathcal{X}\times\C\to\C$ stand for the canonical projections. By Proposition~\ref{prop:laminar}, one can write $\widehat{T}_{\mathcal{X}}=\int_{\Sigma_{\ell}}\widehat{T}_\epsilon\, d\nu_{\ud}(\epsilon)$. As $\widehat{T}_{\mathcal{X}}$ has a continuous potential, $\widehat{T}_{\mathcal{X}}\wedge[\mathcal{C}_{i_j}]$ is admissible. According to Lemma \ref{lm:admissible}, $\widehat{T}_\epsilon\wedge[\mathcal{C}_{i_j}]$ is admissible for $\nu_{\ud}$-a.e. $\epsilon\in\Sigma_{\ell}$ and one can write
$$dd^cg_\lambda(c_{i_j}(\lambda))=(p_2)_*\left(\widehat{T}_{\mathcal{X}}\wedge [\mathcal{C}_{i_j}]\right)=\int_{\Sigma_{\ell}}(p_2)_*\left(\widehat{T}_\epsilon\wedge [\mathcal{C}_{i_j}]\right)\, d\nu_{\ud}(\epsilon)~.$$
Let us set $T_{\epsilon,j}:=(p_2)_*\left(\widehat{T}_\epsilon\wedge [\mathcal{C}_{i_j}]\right)$, as soon as this product is admissible and $T_{\epsilon,j}:=0$ otherwise. This concludes the proof.
\end{proof}

We are now in postion to prove Theorem \ref{cor:TI}.

\begin{proof}[Proof of Theorem \ref{cor:TI}]
First, notice that Lemma~\ref{lm:X} gives $\mathcal{X}_0$, $\mathcal{X}$ and $\pi$ satisfying properties $1$ and $2$. When $k=1$, item $3$ follows easily from Theorem~\ref{tm:laminar}. Indeed, in that case, $I=i_1$ and by construction of $\mathcal{X}$, for any $j\neq i_1$, the critical point $c_j$ is stable (they escape on te whole family). In particular, one has
\[dd^cL\circ\pi=dd^cg_\lambda(c_{i_1}(\lambda))=\int_{\Sigma_\ell}T_{\epsilon,1}\, d\nu_{\ud}(\epsilon)~.\]

We thus assume that $k\geq2$. We want to prove item $3$. For the sake of simplicity, write $\Sigma=\Sigma_{\ell}$ and $\nu=\nu_{\ud}$. Again, as the functions $g_\lambda(c_{i_1}(\lambda)),\ldots,g_\lambda(c_{i_k}(\lambda))$ are continuous, for any $1\leq m\leq k$, the wedge product $dd^cg_\lambda(c_{i_1}(\lambda))\wedge\cdots\wedge dd^cg_\lambda(c_{i_m}(\lambda))$ is admissible. By an easy induction, according to Lemma \ref{lm:admissible} and to Fubini's Theorem, for any $1\leq m\leq k$  the product $T_{\epsilon_1,1}\wedge\cdots\wedge T_{\epsilon_m,m}$ is admissible for $\nu^{\otimes m}$-a.e. $\epsilon=(\epsilon_1,\ldots,\epsilon_m)$ and,
\begin{eqnarray*}
\bigwedge_{j=1}^kdd^cg_\lambda(c_{i_j}(\lambda)) & = & \int_{\Sigma}\left(T_{\epsilon,1}\wedge\bigwedge_{j=2}^kdd^cg_\lambda(c_{i_j}(\lambda))\right)d\nu(\epsilon)\\
& = & \int_{\Sigma}\left(T_{\epsilon_1,1}\wedge\int_{\Sigma}\left(T_{\epsilon_2,2}\wedge\bigwedge_{j=3}^kdd^cg_\lambda(c_{i_j}(\lambda))\right)d\nu(\epsilon_2)\right)d\nu(\epsilon_1)\\
& = & \int_{\Sigma^2}\left(T_{\epsilon_1,1}\wedge T_{\epsilon_2,2}\wedge\bigwedge_{j=3}^kdd^cg_\lambda(c_{i_j}(\lambda))\right)d\nu(\epsilon_2)d\nu(\epsilon_1)\\
& & \\
& \vdots & \\
& = & \int_{\Sigma^k}\left(T_{\epsilon_1,1}\wedge\cdots\wedge T_{\epsilon_k,k}\right)d\nu(\epsilon_1)\cdots d\nu(\epsilon_k)~.
\end{eqnarray*}
By Proposition \ref{propTk}, this yields item $3$, letting $T_{\epsilon_1,1}\wedge \cdots \wedge T_{\epsilon_k,k}:=0$ if it is not admissible.

~

\par Let us now prove item $4$. When $\widehat{T}_{\epsilon}\wedge[\mathcal{C}_{i_j}]$ is admissible, its support is included in
$$\{(\lambda,z)\in\mathcal{X}\times\D(0,R) \, ; \ z\in\K_{\epsilon,\lambda}\}\cap\mathcal{C}_{i_j}=\{(\lambda,z)\in\mathcal{X}\times\C \, ; \ c_{i_j}(\lambda)\in\K_{\epsilon,\lambda}\}~.$$
As a consequence, $\supp(T_{\epsilon,j})\subset\{\lambda\in\mathcal{X} \, ; \ c_{i_j}(\lambda)\in\K_{\epsilon,\lambda}\}$. Let $\mathcal{U}$ be a connected component of the interior of $\{G_I\circ\pi=0\}$. Then $\mathcal{U}$ is a stable component, i.e. the sequences $\{\lambda\mapsto P_\lambda^n(c_{i_j}(\lambda))\}_{n\geq1}$ are normal families in $\mathcal{U}$ as families of holomorphic functions of the parameter, for $1\leq j\leq k$. This implies, for all $1\leq j\leq k$, the existence of $\epsilon_{0,j}\in\Sigma_{d-q}$ such that $c_{i_j}(\lambda)\in\K_{\epsilon_{0,j},\lambda}$ for any $\lambda\in\overline{\mathcal{U}}$, since otherwise the orbit of $c_{i_j}(\lambda)$ would have to lie in the attracting basin of $\infty$ for some $\lambda\in\overline{\mathcal{U}}$, contradicting our assumption that $\mathcal{U}\subset\{G_I\circ\pi=0\}$. Hence
\begin{eqnarray*}
\left\langle(dd^cG_I\circ\pi)^k,\mathbf{1}_{\overline{\mathcal{U}}}\right\rangle & = & \int_{\Sigma^k}\left\langle \bigwedge_{j=1}^kT_{\epsilon_j,j},\mathbf{1}_{\overline{\mathcal{U}}}\right\rangle\,d\nu^{\otimes k}(\epsilon)\\
& \leq & \|T_{\epsilon_{0,1},1}\wedge \cdots\wedge T_{\epsilon_{0,k},k}\|_{\overline{\mathcal{U}}}\cdot\prod_{j=1}^k\nu(\{\epsilon_{0,j}\})=0~,
\end{eqnarray*}
which concludes the proof.
\end{proof}

\section{The bifurcation measure does not charge boundary components}
In the present Section, we focus on the proof of Theorem \ref{lmboundary}. Namely, we prove that, as in the quadratic family, given any connected component $U$ of the interior of the connectedness locus $\mathcal{C}_d$, the bifurcation measure doesn't give mass to the boundary of $U$. The proof of Theorem \ref{lmboundary} uses the continuity of the Julia set at some specific parameters due to Douady \cite{Douady}, convergence of invariant line fields established by McMullen \cite{McMullen}, as well as a precise dynamical description of $\mu_\bif$-a.e. polynomial due to Dujardin and Favre \cite{favredujardin}.

\subsection{Invariant line fields and the Caratheodory topology}
For the material of the present section, we refer to \cite{McMullen}.

\begin{definition}
Let $U\subset\C$ be an open set. A \emph{measurable line field} on a Borel set of positive area $E\subset U$ is a Beltrami coefficient
$$\nu=\nu(z)\frac{d\bar z}{d z}$$
where $\nu(z)$ is a measurable map on $U$ with $|\nu(z)|=1$ if $z\in E$ and $\nu(z)=0$ otherwise. Let $V\subset\C$ be another open set. We say that the line field $\nu$ is \emph{invariant} by a holomorphic map $f:U\to V$, or $f$-\emph{invariant}, if $f^*\nu=\nu$ on $U\cap V$.
\end{definition}

\medskip

Let us consider a sequence $(V_n,x_n)$ of pointed topological disks of $\p^1$. We say that $(V_n,x_n)$ converges to $(V,x)$ in the Caratheodory topology if
\begin{enumerate}
\item $x_n\to x$ as $n\to\infty$,
\item for all compact set $K\subset V$, there exists $N\geq1$ such that $K\subset V_n$ for all $n\geq N$,
\item for any open set $U\subset\p^1$ containing $x$, if there exists $N\geq1$ such that $U\subset V_n$ for all $n\geq N$, then $U\subset V$.
\end{enumerate}
If $(U_n,x_n)\to(U,x)$ and $(V_n,y_n)\to(V,y)$ in the Caratheodory topology and if $f_n:U_n\to V_n$ is a sequence of holomorphic maps satisfying $f_n(x_n)=y_n$ whic converges uniformly on compact subsets of $U$ to $f:U\to V$ holomorphic with $f(x)=y$, we say that $f_n:(U_n,x_n)\to(V_n,y_n)$ \emph{converges in the Carath\'eodory topology} to $f:(U,x)\to(V,y)$.

~

\par Recall the following definition (see \cite[\S 5.6]{McMullen}).
\begin{definition}
We say that a sequence $\nu_n\in L^\infty(V,\C)$ \emph{converges in measure} to $\nu\in L^{\infty}(V,\C)$ on $V$ if for all compact $K\Subset V$ and all $\varepsilon>0$,
$$
\lim_{n\to\infty}\textup{Area}\left(\{z\in K \ ; \ |\nu_n(z)-\nu(z)|>\varepsilon\}\right)=0~.$$
\end{definition}
\noindent According to \cite[Proposition 2.37.3]{wagschal}, a bounded sequence $\nu_n\in L^\infty(\C,\C)$ admits a subsequence which converges in measure if and only if it is a Cauchy sequence in measure, i.e. for any compact $K\Subset\C$ and for any $\delta,\epsilon>0$, there exists $n\geq1$ such that
$$\textup{Area}\left(\{ z\in K\ : \ |\nu_p(z)-\nu_{q}(z)|>\delta\}\right)\leq \epsilon~,$$
for any $p,q\geq n$. 

~

In what follows, we shall use the following result of McMullen (see \cite[Theorem 5.14]{McMullen}).
\begin{theorem}[McMullen]\label{tm:lf}
Let $f_n:(U_n,x_n)\to(V_n,y_n)$ be a sequence of non-constant holomorphic maps between disks. Assume that $f_n$ converges in the Caratheodory topology to a non-constant holomorphic map $f:(U,x)\to(V,y)$ . Assume in addition that there exists a measurable $f_n$-invariant line field $\nu_n$ which converges in measure to $\nu$ on $V$. Then $\nu$ is a measurable $f$-invariant line field.
\end{theorem}

As a consequence, we immediately have $\textup{Area}(\supp(\nu))>0$.

\subsection{Some pathologic filled-in Julia sets of positive area}
In the present section, we aim at proving that, for polynomials belonging to the boundary of queer components where $\K_{c,a}=\J_{c,a}$, the filled-in Julia set has positive area. Precisely, we prove the following.

\begin{theorem}
Let $\mathcal{U}\subset\C^{d-1}$ be a connected component of the interior of $\mathcal{C}_d$. Assume that there exists a parameter $(c,a)\in \mathcal{U}$ such that $P_{c,a}$ has only repelling cycles and let $(c_0,a_0)\in\partial \mathcal{U}$. Then, either $\textup{Area}(\J_{c_0,a_0})>0$, or $\K_{c_0,a_0}$ has non-empty interior.\label{tm:dim}
\end{theorem}

\begin{proof}
As there exists $(c,a)\in \mathcal{U}$ such that $P_{c,a}$ has only repelling cycles, one has $\J_{c,a}=\K_{c,a}$. Moreover, as $\mathcal{U}\subset\mathcal{C}_d$, it is a stable component. This implies that $\J_{c,a}=\K_{c,a}$ for all $(c,a)\in \mathcal{U}$. In particular, $\mathcal{U}$ is not a hyperbolic component. By \cite[Theorem E]{MSS}, for any $(c,a)\in \mathcal{U}$, there exists a $P_{c,a}$-invariant line field $\nu_{c,a}$ which is supported on the Julia set $\J_{c,a}$ of $P_{c,a}$, i.e. $\nu_{c,a}\in L^\infty(\C,\C)$ satisfies $P_{c,a}^*\nu_{c,a}=\nu_{c,a}$ and there exists a Borel set $E_{c,a}\subset\J_{c,a}$ of positive area such that $|\nu_{c,a}(z)|=1$ for all $z\in E_{c,a}$, and $\nu_{c,a}(z)=0$ for all 
$z\notin E_{c,a}$.

~

\par Let us briefly recall how, in the present case, one can build this invariant line field. Let $(c_1,a_1)\in \mathcal{U}$ be a base point that we have chosen and let $\psi_{c,a}$ stand for the B\"ottcher coordinate of $\infty$ of $P_{c,a}$. The family of analytic maps
$$\phi_{c,a}(z):=\psi^{-1}_{c,a}\circ \psi_{c_1,a_1}(z), \ z\in \C\setminus\J_{c_1,v_1}~,$$
defines a conformal holomorphic motion $\mathcal{U}\times\left(\C\setminus\J_{c_1,a_1}\right)\to\C$ which satisfies
$$ \phi_{c,a}\circ P_{c_1,a_1}(z)=\psi_{c,a}^{-1}\circ \psi_{c_1,a_1}(P_{c_1,a_1}(z))=\psi_{c,a}^{-1}( \psi_{c_1,a_1}(z)^d)=P_{c,a}\circ \phi_{c,a}(z)~.$$
By the $\lambda$-Lemma, it extends as a quasiconformal holomorphic motion $\phi:\mathcal{U}\times\p^1\to\p^1$ such that $\phi_{c,a}$ conjugates $P_{c_1,a_1}$ to $P_{c,a}$ on $\C$. Let $\mu_{c,a}$ be the Beltrami form on $\C$ satisfying
$$\overline{\partial} \phi_{c,a}^{-1}=\mu_{c,a}\circ\partial \phi_{c,a}^{-1}$$
%$$\frac{\partial \phi_{c,a}^{-1}}{\partial \overline{z}}(z)d\bar z=\mu_{c,a}(z)\cdot\frac{\partial \phi_{c,a}^{-1}}{\partial z}(z)dz$$
almost everywhere on $\C$. Then $\supp(\mu_{c,a})\subset\J_{c,a}$. If $\textup{Area}(\supp(\mu_{c_2,a_2}))=0$ for some $(c_2,a_2)\in\mathcal{U}$, it would also be the case for all $(c,a)\in\mathcal{U}$. By the above construction, the maps $\phi_{c,a}$ would be a quasi-conformal homeomorphism which is holomorphic almost everywhere, i.e. $\phi_{c,a}\in\textup{Aut}(\C)$. This contradicts the fact that the family $(P_{c,a})_{(c,a)\in\C^{d-1}}$ is a finite ramified cover of the moduli space $\mathcal{P}_d$. Hence the Beltrami form defined by
$$\nu_{c,a}:=\left\{\begin{array}{cl}
\frac{\mu_{c,a}(z)}{|\mu_{c,a}(z)|}\cdot\frac{d\bar z}{dz}& \text{if }\mu_{c,a}(z)\neq0~,\\
0 & \text{otherwise}
\end{array}
\right.$$
defines an invariant line field for $P_{c,a}$.

~

Let us now proceed by contradiction, assuming that, for some $(c_0,a_0)\in\partial \mathcal{U}$, one has $\J_{c_0,a_0}=\K_{c_0,a_0}$ and  $\K_{c_0,a_0}$ has Lebesgue measure zero. According to \cite[Corollaire 5.2]{Douady}, the map $(c,a)\mapsto\J_{c,a}$ is continuous at $(c_0,a_0)$. By \cite[Corollary 6.5.2]{ransford}, for any $(c,a)\in\mathcal{C}_d$, the compact set $\K_{c,a}$ contains $0$ and
$$\K_{c,a}\subset\overline{\D(0,4\sqrt[d-1]{d})}~.$$
By Montel's Theorem, the family $(\psi_{c,a}^{-1})_{(c,a)\in \mathcal{U}}$ is a normal family and, for all $z\in \C\setminus\D(0,4d)$, one has $\lim_{(c,a)\to(c_0,a_0)}\psi_{c,a}^{-1}(z)=\psi_{c_0,a_0}^{-1}(z)$, where $\psi_{c_0,a_0}$ is the B\"ottcher coordinate at $\infty$ of $P_{c_0,a_0}$. In particular, the family $(\psi^{-1}_{c,a})_{(c,a)\in\mathcal{U}}$ converges locally uniformly to $\psi_{c_0,a_0}^{-1}$ on $\C\setminus\overline{\D}$ as $(c,a)\to(c_0,a_0)$. Hence, for $R>0$ big enough, the closed topological disk
$$(\overline{\psi_{c,a}^{-1}(\p^1\setminus\overline{\D(0,R)})},\infty)$$
converges to the closed topological disk
$$(\overline{\psi_{c_0,a_0}^{-1}(\p^1\setminus\overline{\D(0,R)})},\infty)$$
in the Caratheodory topology, as $(c,a)\to (c_0,a_0)$, and for all $(c,a)\in\mathcal{U}\cup\{(c_0,a_0)\}$,
$$\overline{\D(0,4\sqrt[d-1]{d})}\cap\psi_{c,a}^{-1}(\p^1\setminus\overline{\D(0,R)})=\emptyset~.$$
If we set
$$U_{c,a}:=\{z\in\C \, ; \ g_{c,a}(z)<\log R\}=\psi_{c,a}^{-1}(\C\setminus\overline{\D(0,R)})$$
and $V_{c,a}:=P_{c,a}(U_{c,a})=\{z\in\C \, ; \ g_{c,a}(z)<d\log R\}$, the open sets $U_{c,a}$ and $V_{c,a}$ are topological disks and
$(U_{c,a},0)\to(U_{c_0,a_0},0)$ and $(V_{c,a},a^d)\to(V_{c_0,a_0},a_0^d)$ in the Caratheodory topology as $(c,a)\to(c_0,a_0)$.

~

As $\J_{c_n,a_n}$ converges in the Hausdorff topology to $\J_{c_0,a_0}$, one has
$$0\leq\limsup_{n\to\infty}\textup{Area}(\J_{c_n,a_n})\leq\textup{Area}(\J_{c_0,a_0})=0~,$$
which means that $\lim_{n\to\infty}\textup{Area}(\J_{c_n,a_n})=\textup{Area}(\J_{c_0,a_0})=0$.

 Let $K\Subset\C$ be a compact subset and $\delta,\epsilon>0$. As $\supp(\nu_{c_n,a_n})\subset\J_{c_n,a_n}$, there exists $n\geq1$ such that $\textup{Area}(\supp(\nu_{c_p,a_p}))\leq \epsilon/2$ for all $p\geq n$. Let now $p,q\geq n$. Then
$$\{ z\in K\ : \ |\nu_{c_p,a_p}(z)-\nu_{c_q,a_q}(z)|>\delta\}\subset\supp(\nu_{c_p,a_p})\cup\supp(\nu_{c_q,a_q})~,$$
hence 
$$\textup{Area}\left(\{ z\in K\ : \ |\nu_{c_p,a_p}(z)-\nu_{c_q,a_q}(z)|>\delta\}\right)\leq \epsilon~.$$
The sequence $(\nu_{c_n,a_n})$ is thus a Cauchy sequence in measure and we can find a sequence $\{(c_n,a_n)\}_{n\geq1}$ (extracted from the previous one) which converges to $(c_0,a_0)$ as $n$ tends to $\infty$ and such that $\nu_{c_n,a_n}$ converges in measure to some function $\nu_0\in L^\infty(\C,\C)$.

~

\par Finally, since $\left(U_{c_n,a_n},0\right)\to \left(U_{c_0,a_0},0\right)$ and $\left(V_{c_n,a_n},a_n^d\right)\to \left(V_{c_0,a_0},a_0^d\right)$ converge in the Carath\'eodory topology and since $P_{c_n,a_n}$ converges uniformly on compact subsets of $\C$ to $P_{c_0,a_0}$, we may apply McMullen Theorem~\ref{tm:lf} to the sequences $(\nu_{c_n,a_n})$ and
$$P_{c_n,a_n}:\left(U_{c_n,a_n},0\right)\to \left(V_{c_n,a_n},a_n^d\right)~.$$
The conclusion is that $\nu_0$ is a $P_{c_0,a_0}$-invariant line field on $\J_{c_0,a_0}$. In particular, $\J_{c_0,a_0}$ must have positive area, since it carries an invariant line field. This is a contradiction.
\end{proof}

\subsection{Proof of Theorem \ref{lmboundary}}
Recall that there exists a Borel set $\mathcal{B}\subset\partial_S \mathcal{C}_d$ of full measure for the bifurcation measure $\mu_\bif$ and such that for all $(c,a)\in\mathcal{B}$, (see Theorem \ref{tm:mupp})
\begin{itemize}
\item all cycles of $P_{c,a}$ are repelling,
\item the orbit of each critical points are dense in $\J_{c,a}$,
\item $\K_{c,a}=\J_{c,a}$ is locally connected and $\dim_H(\J_{c,a})<2$.
\end{itemize}
Let $\mathcal{U}\subset\C^{d-1}$ be a connected component of the interior of $\mathcal{C}_d$. It is a stable component and we treat separately two cases. Assume first that there exists $(c,a)\in \mathcal{U}$ such that $P_{c,a}$ has at least one non-repelling cycle. As the hypersurface $\Per_n(e^{i\theta})$ lies in the bifurcation locus for any $n\geq1$ and $\theta\in\R$, the polynomial $P_{c,a}$ has at least one attracting periodic point $z(c,a)$ and it can be followed holomorphically on $\mathcal{U}$. Hence it extends as a continuous map $z:\overline{\mathcal{U}}\to\C$ such that $z(c,a)$ is periodic for $P_{c,a}$ for all $(c,a)\in\overline{\mathcal{U}}$. In particular, for all $(c,a)\in\partial \mathcal{U}$, the polynomial $P_{c,a}$ admits a non-repelling periodic point. In particular, $\mathcal{B}\cap\partial \mathcal{U}=\emptyset$ by Theorem \ref{tm:mupp}, hence $\mu_\bif(\partial \mathcal{U})=0$.

~

\par Assume now that there exists $(c,a)\in \mathcal{U}$ such that all the periodic points of $P_{c,a}$ are repelling. Then, according to \cite[Theorem E]{MSS}, for any $(c,a)\in \mathcal{U}$, $P_{c,a}$ carries an invariant line field on its Julia set, $\J_{c,a}=\K_{c,a}$ and $\textup{Area}(\J_{c,a})>0$. Let $(c_0,a_0)\in\partial \mathcal{U}$, as $(c,a)\in \mathcal{U}\to (c_0,a_0)$, either all the cycles of $P_{c,a}$ remain repelling, or at least one becomes non-repelling. One thus has the following dichotomy:
\begin{enumerate}
\item all cycles of $P_{c_0,a_0}$ are repelling and thus $\J_{c_0,a_0}=\K_{c_0,a_0}$, or
\item there exists one cycle of $P_{c_0,a_0}$ which is non-repelling.
\end{enumerate}
In the first case, according to Theorem \ref{tm:dim}, one has $\textup{Area}(\J_{c_0,a_0})>0$, hence $(c_0,a_0)\not\in\mathcal{B}$. In the second case, according to Theorem \ref{tm:mupp}, one has $(c_0,a_0)\not\in\mathcal{B}$. We thus have proved that, in any case, $\partial \mathcal{U}\cap\mathcal{B}=\emptyset$ and $\mu_\bif(\partial \mathcal{U})=0$.

\section{Distribution of the hypersurfaces $\Per_n(w)$ for any $w$}

The present section is dedicated to the proof of Theorem \ref{tmconv}. In a first time, we recall  the definition of the hypersurface $\Per_n(w)$ and a result concerning limits in the sense of currents of these hypersurfaces due to Bassanelli and Berteloot \cite{BB2}.

\subsection{The hypersurfaces $\Per_n(w)$}

In what follows, we shall use the following (see \cite{silverman,milnor3}):
\begin{theorem}[Milnor, Silverman]\label{tmpern}
For any $n\geq1$, there exists a polynomial $p_n:\C^d\to\C$ such that for any $(c,a)\in\C^{d-1}$ and any $w\in\C$,
\begin{enumerate}
\item  if $w\neq1$, then $p_n(c,a,w)=0$ if and only if $P_{c,a}$ has a cycle of exact period $n$ and multiplier $w$,
\item otherwise, $p_n(c,a,1)=0$ if and only if there exists $q\geq1$ such that $P_{c,a}$ has a cycle of exact period $n/q$ and multiplier $\eta$ a primitive $q$-root of unity.
\end{enumerate}
\end{theorem}

\par\noindent We now define a hypersurface by letting
$$\Per_n(w):=\{(c,a)\in\C^{d-1} \ | \ p_n(c,a,w)=0\}~,$$
for $n\geq1$ and $w\in\C$. We also shall set $L_{n,w}(c,a):=\log|p_n(c,a,w)|$ so that
$$[\Per_n(w)]=dd^c_{c,a}L_{n,w}~.$$
Bassanelli and Berteloot show the following we will rely on (see for instance \cite[Propositions 3.2 $\&$ 3.3]{BB2}).

\begin{proposition}[Bassanelli-Berteloot]
Pick $w\in\C$. The sequence $(d^{-n}L_{n,w})$ is reatively compact in $L^1_{\textup{loc}}(\C^{d-1})$. Let $\varphi$ be any limit of the sequence $(d^{-n}L_{n,w})$. Then $\varphi$ is a \emph{p.s.h} function which satisfies
\begin{itemize}
\item $\varphi\leq L$ on $\C^{d-1}$,
\item $\varphi=L$ on hyperbolic components. In particular, $\varphi\not\equiv-\infty$.
\end{itemize}\label{prop:BB}
\end{proposition}

\subsection{Proof of Theorem \ref{tmconv}}\label{sec:distrib}

We are now in position to prove Theorem \ref{tmconv}. Let us first remark that, since the natural projection $\pi:\C^{d-1}\to\mathcal{P}_d$ defined by $\pi(c,a)=\{P_{c,a}\}$ is $d(d-1)$-to-$1$, it is sufficient to prove that equidistribution holds in the family $(P_{c,a})_{(c,a)\in\C^{d-1}}$.

~

Pick any $w\in\C$, let $\varphi$ be any $L^1_{\textup{loc}}$-limit of the sequence $(d^{-n}L_{n,w})$ and let $(d^{-n_k}L_{n_k,w})_{k\geq0}$ converge to $\varphi$ in $L^1_{\textup{loc}}$. By Proposition \ref{prop:BB}, we have
\begin{enumerate}
\item $\varphi\leq L$ on $\C^{d-1}$,
\item $\varphi=L$ on hyperbolic components and in particular, $\varphi\not\equiv-\infty$,
\end{enumerate}
Our strategy is to make inductively use of the comparison principle which is established in Section \ref{sec:comp} to prove that $\varphi=L$. First, let us define an open set $U$ by setting
$$U:=\bigcup_{k=0}^{d-2}U_k, \ \text{ with } \ U_k:=\bigcup_{I}\bigcup_{\tau\in\mathfrak{S}_{d-1-k}}U_{I,\tau}~,$$
where $I=(i_1,\ldots,i_k)$ ranges over $k$-tuples with $0\leq i_1<\cdots<i_k\leq d-2$ and where $U_{I,\tau}$ are the open sets defined in Section \ref{sec:structure}.

\begin{claim}
$U$ is an open and dense subset of $\C^{d-1}\setminus\mathcal{C}_d$.
\end{claim}
 
 We may prove that $L=\varphi$ on $U$. As $L$ and $\varphi$ are psh and as $L$ is continuous, this yields $L=\varphi$ on $\C^{d-1}\setminus\mathcal{C}_d$. Indeed, if $(c,a)\in\C^{d-1}\setminus\mathcal{C}_d$, there exists $U\ni(c_n,a_n)\to(c,a)$ and
$$\varphi(c,a)\leq L(c,a)=\limsup_{n\to\infty}L(c_n,a_n)=\limsup_{n\to\infty}\varphi(c_n,a_n)\leq\varphi(c,a)~.$$

\smallskip

Pick now $(c,a)\in U$ and let $0\leq k\leq d-2$ be the number of critical points of $P_{c,a}$ with bounded orbit. If $k=0$, then $(c,a)\in\mathcal{E}:=\C^{d-1}\setminus\bigcup_j\mathcal{B}_j$. Since $\mathcal{E}$ is in an hyperbolic component, hence $\varphi(c,a)=L(c,a)$. In particular, $\varphi=L$ on $U\cap \mathcal{E}$.

Assume now that $1\leq k\leq d-2$ and that $\varphi=L$ on
$$V_k:=U\setminus \bigcup_{0\leq i_1<\cdots<i_{k-1}\leq d-2}\bigcap_{j=1}^{k-1}\mathcal{B}_{i_j}~,$$
i.e. on the locus on $U$ where at least $d-k$ critical points escape. Since $k$ critical points of $P_{c,a}$ don't escape, $(c,a)\in\{G_I=0\}\cap U_{I,\tau}$ for some $k$-tuple $I$ and some $\tau\in\mathfrak{S}_{d-1-k}$. Let us remark that $\{G_I>0\}\cap U_{I,\tau}$ is contained in the aforementioned open set $V_k$, so that $\varphi=L$ on $\{G_I>0\}\cap U_{I,\tau}$. According to Theorem \ref{cor:TI}, there exists a $k$-dimensional manifold $\mathcal{X}$, an analytic set $\mathcal{X}_0$ and a finite holomorphic mapping $\pi:\mathcal{X}\to\mathcal{X}_0$ such that
\begin{itemize}
\item $\mathcal{X}$ has dimension $k$,
\item $\{G_I\circ\pi=0\}\Subset\mathcal{X}$, in particular $\varphi\circ\pi=L\circ\pi$ on $\mathcal{X}\setminus\{G_I\circ\pi=0\}$,
\item $(dd^cL\circ\pi)^k$ is a finite measure on $\mathcal{X}$ supported by $\partial\{G_I\circ\pi=0\}$,
\item for any connected component $\mathcal{U}$ of the interior of $\{G_I\circ\pi=0\}$,
\begin{center}
$(dd^cL\circ\pi)^k(\partial\mathcal{U})=0$.
\end{center}
\end{itemize}
To apply the comparison Theorem \ref{tmcompXav}, it remains to justify the existence of a smooth form on $\mathcal{X}$ which is K\"ahler outside an analytic subset of $\mathcal{X}$. Let $\omega:=dd^c\|(c,a)\|^2$ be the standard K\"ahler form on $\C^{d-1}$. Then the function $\lambda\mapsto\|\pi(\lambda)\|^2$ is psh and smooth on $\mathcal{X}$. Moreover, the form
$$\omega_\mathcal{X}:=dd^c\|\pi(\lambda)\|^2=\pi^*\left(\omega|_{\mathcal{X}_0}\right)$$
is K\"ahler on $\mathcal{X}\setminus\mathcal{Z}$, where $\mathcal{Z}$ is the strict analytic set of parameters $\lambda\in\mathcal{X}$ such that $D_\lambda\pi$ doesn't have maximal rank. By Theorem \ref{tmcompXav}, one has $\varphi\circ\pi=L\circ\pi$ on $\mathcal{X}$. In particular, $\varphi(c,a)=L(c,a)$. We thus have shown that $\varphi=L$ on the open set
$$V_{k+1}=U\setminus \bigcup_{0\leq i_1<\cdots<i_{k}\leq d-2}\bigcap_{j=1}^{k}\mathcal{B}_{i_j}~.$$
By a finite induction on $k$, we have $\varphi=L$ on $U$, hence on $\C^{d-1}\setminus\mathcal{C}_d$, as explained above.

~

The final step of the proof goes essentially the same way. According to Theorem \ref{lmboundary},
\begin{itemize}
\item $L$ is continuous and psh on $\C^{d-1}$ and the bifurcation measure
$$(dd^cL)^{d-1}=\mu_\bif$$
is supported on $\partial_S\mathcal{C}_d\subset\partial\mathcal{C}_d$,
\item for any connected component $\mathcal{U}$ of $\mathring{\mathcal{C}}_d$, $(dd^cL)^{d-1}(\partial\mathcal{U})=0$,
\item $\varphi\leq L$ and $\varphi=L$ on $\C^{d-1}\setminus\mathcal{C}_d$.
\end{itemize}
By Theorem \ref{tmcompXav}, this yields $\varphi=L$. Since this works for any $L^1_\textup{loc}$ limit $\varphi$ of the sequence $(d^{-n}L_{n,w})$, this means that $(d^{-n}L_{n,w})$ converges in $L^1_\textup{loc}$ to $L$, which ends the proof.

~

It now only remains to prove the Claim.

\begin{proof}[Proof of the Claim]
The openess is obvious by continuity of the maps $(c,a)\mapsto g_{c,a}(c_j)$, $0\leq j\leq d-2$. For $0\leq k\leq d-2$, $I=(i_1,\ldots,i_k)$ with $0\leq i_1<\cdots <i_k\leq d-2$ and $\tau\in\mathfrak{S}_{d-k-1}$, we let $V_{k,I,\tau}\subset\C^{d-1}$ be the set
\begin{equation*}
V_{k,\tau}:=\{G_I<g_{c,a}(c_{j_{\tau(1)}})\leq \cdots\leq g_{c,a}(c_{j_{\tau(d-k-1)}})\}~,
\end{equation*}
where $\{j_1,\ldots,j_{d-k-1}\}=I^c$, so that $\bigcup_{k,I,\tau}V_{k,I,\tau}=\C^2\setminus\mathcal{C}_d$ and $U_{I,\tau}\subset V_{k,I,\tau}$. It is sufficient to prove that $U_{I,\tau}$ is dense in $V_{k,I,\tau}$ for any $k$ and any $\tau$ to conclude.

Let now $0\leq k\leq d-2$ and $\tau\in\mathfrak{S}_{d-1}$ be fixed. Assume by contradiction that $V_{k,I,\tau}\setminus U_{I,\tau}$ contains an open set $\Omega$ of $\C^{d-1}\setminus\mathcal{C}_d$. Then, there exists $1\leq l\leq d-k-2$ so that the map
$$\phi_l:(c,a)\mapsto g_{c,a}(c_{j_{\tau(l+1)}})-g_{c,a}(c_{j_{\tau(l)}})$$
is constant equal to $0$ on $\Omega$. On the other hand, as $\Omega\subset W:=\{g_{c,a}(c_{j_{\tau(l)}})>0\}\cap\{g_{c,a}(c_{j_{\tau(l+1)}})>0\}$, the functions $g_{c,a}(c_{j_{\tau(l)}})$ and $g_{c,a}(c_{j_{\tau(l+1)}})$ are pluriharmonic on $W$, the function $\phi_l$ is pluriharmonic on the connected component $V$ of $W$ constaining $\Omega$ and vanishes on $\Omega$. As $V$ is connected and $\Omega\subset V$ is a non-empty open set, $\phi_l\equiv0$ on $V$, hence on $\overline{V}$, by continuity of $\phi_l$. This means that the open set $V$ is a connected component of $\{g_{c,a}(c_{j_{\tau(l)}})>0\}$.

To conclude the proof of the Claim, we just have to remark that, due to Lemma~\ref{lm:connected}, we have shown that $g_{c,a}(c_{j_{\tau(l)}})\equiv g_{c,a}(c_{j_{\tau(l+1)}})$ on $\C^{d-1}$, which is impossible, by Theorem \ref{tmBH}.
\end{proof}

\medskip

\noindent\textbf{Acknowledgement} We would like to thank Gabriel Vigny for many interesting and helpful discussions. We also would like to thank Vincent Guedj for interesting discussions concerning the comparison principle. The author also thanks the referee whose comments greatly helped to improve the presentation of the paper.

\bibliographystyle{short}
\bibliography{biblio}
\end{document}